\RequirePackage{fix-cm}
\documentclass[smallcondensed]{nsvjour3}     
\smartqed  
\usepackage{lineno}

\usepackage{amssymb, amsmath, amscd, amsthm}
\usepackage{thmtools,thm-restate} 
\usepackage{url}
\usepackage[english]{babel}
\usepackage[utf8]{inputenc}
\usepackage[margin=1in]{geometry}
\usepackage[colorinlistoftodos]{todonotes}
\usepackage{dsfont}
\usepackage{xparse}
\usepackage{amsfonts}
\usepackage[all]{xy}
\usepackage{graphicx}
\usepackage{epstopdf}
\usepackage{subcaption}
\usepackage{verbatim}
\usepackage{mathtools}
\usepackage{faktor}
\usepackage[normalem]{ulem}
\usepackage{xfrac}

\usepackage{array}
\usepackage{tabularx}
\usepackage{booktabs}
\usepackage{multirow}
\usepackage{makecell}

\newcommand\Tstrut{\rule{0pt}{2.6ex}}         
\newcommand\Bstrut{\rule[-1.4ex]{0pt}{0pt}}   

\newtheorem{thm}{Theorem}
\newtheorem{cor}{Corollary}[section]
\newtheorem{lem}{Lemma}[section]
\newtheorem{prop}{Proposition}[section]
\theoremstyle{definition}
\newtheorem{defn}{Definition}[section]
\theoremstyle{remark}
\newtheorem{rem}{Remark}[section]
\numberwithin{equation}{section}
\newtheorem{ex}{Example}[section]

\ExplSyntaxOn
\NewDocumentCommand{\cycle}{ O{\;} m }
{
	(
	\alec_cycle:nn { #1 } { #2 }
	)
}

\seq_new:N \l_alec_cycle_seq
\cs_new_protected:Npn \alec_cycle:nn #1 #2
{
	\seq_set_split:Nnn \l_alec_cycle_seq { , } { #2 }
	\seq_use:Nn \l_alec_cycle_seq { #1 }
}
\ExplSyntaxOff

\newcommand{\R}{\mathbb{R}}
\newcommand{\N}{\mathbb{N}}

\newcommand{\id}{\mathrm{id}}

\newcommand{\sgn}{\mathrm{sgn}}

\newcommand{\PMe}{\mathrm{PM}}
\newcommand{\Perm}{\mathrm{Perm}}

\begin{document}

\title{On the finite representation of linear group equivariant operators via permutant measures \thanks{This research has been partially supported by INdAM-GNSAGA.}
}


\author{Giovanni Bocchi \and
        Stefano Botteghi \and
        Martina Brasini \and
        Patrizio Frosini \and
        Nicola Quercioli
}


\institute{Giovanni Bocchi \at Department of Environmental Science and Policy, University of Milan, Italy \\
		\email{giovanni.bocchi1@unimi.it}           
		\and
		Stefano Botteghi \at Department of Mathematics, University of Bologna, Italy \\
		\email{stefano.botteghi@studio.unibo.it}           
		\and
		Martina Brasini \at Department of Mathematics, University of Bologna, Italy \\
		\email{martina.brasini@studio.unibo.it}           
		\and
		Patrizio Frosini (Corresponding author)\at Department of Mathematics, University of Bologna, Italy \\
		\email{patrizio.frosini@unibo.it}           
		\and
		Nicola Quercioli \at ENEA, Centro Ricerche Bologna, Italy \\
		\email{nicola.quercioli@enea.it}           
		\and
	}


\maketitle

\vskip -2cm

\begin{abstract}
Recent advances in machine learning have highlighted the importance of using group equivariant non-expansive operators for building neural networks in a more transparent and interpretable way.
An operator is called equivariant with respect to a group if the action of the group commutes with the operator.
Group equivariant non-expansive operators can be seen as multi-level components that can be joined and connected in order to form neural networks by applying the operations of chaining, convex combination and direct product.
In this paper we prove that each linear $G$-equivariant non-expansive operator (GENEO) can be produced by a weighted summation associated with a suitable permutant measure, provided that the group $G$ transitively acts on a finite signal domain.
This result is based on the Birkhoff–von Neumann decomposition of doubly stochastic matrices and some well known facts in group theory.
Our theorem makes available a new method to build all linear GENEOs with respect to a transitively acting group in the finite setting.
This work is part of the research devoted to develop a good mathematical theory of GENEOs, seen as relevant components in machine learning. \end{abstract}



\section*{Introduction}
The development of new mathematical approaches and results for deep learning is an important goal in the present time.
In particular, new methods and theories are requested in order to understand and control the behavior of neural networks.
In this line of research, the scientific community is devoting more and more attention to the use of equivariant operators, since they appear to be of great importance for the progress of machine learning~\cite{AnRoPo16,cohen2016group,worrall2017harmonic}.
We recall that an operator is called equivariant with respect to a group if the action of the group commutes with
the operator.
{For example, the operator associating each regular function $f:\R^n\to\R$ to its Laplacian $\Delta f$ commutes with each Euclidean isometry of $\R^n$.
The property of equivariance} is of use when we want to mimic the behavior of observers and agents that are known to respect some symmetries. In fact, the use of equivariant operators allows us to inject pre-existing knowledge into the system, thus increasing our control on the costruction of neural networks~\cite{BeCoVi13}.
Furthermore, invariant and non-expansive operators can be used to reduce data variability~\cite{Ma12,Ma16}, and in the last years equivariant transformations have been studied for learning symmetries~\cite{ZVERP15,AERP19}.

{Due to the relevance of group actions in deep learning, geometry is giving a significant contribution to the study of AI~\cite{BrBrLeSzVa17}.
Geometric Deep Learning is indeed trying to produce a geometric unification of several approaches to machine learning, focusing on the concepts of symmetry and invariance.
At the intersection between this research field and Topological Data Analysis, it has been proposed to extend
the study of the geometry of the space of data to the
study of the geometry of the space of the observers/agents that elaborate the data~\cite{Fr16,FrJa16}. This idea is both natural and relevant, since the interpretation of data
depends on the chosen observers, and
the approximation of the agents requires the knowledge of the topological and geometric properties of the space such agents belong to, including connectivity, convexity, compactness, curvature and so on.

Recently, the study of group equivariant non-expansive operators (GENEOs) has been proposed in~\cite{Bergomi2019} as an interesting topic in machine learning, since
these operators model the concept of data observer. Furthermore, they could be seen as multi-level components that can be joined and connected to form neural networks by applying the operations of chaining, convex combination and direct product, opening the path to a new kind of explainable ``geometric agent engineering''. Therefore, the analysis of the topology and geometry of the spaces of GENEOs could lead to new theoretical results for building neural networks in a more transparent and interpretable way. For example, the use of a suitable topological setting allowed us to prove the compactness of the space of all GENEOs, under the assumption that the space of data is compact ~\cite{Bergomi2019}. This result has relevant practical consequences, since it guarantees the finite approximability of the spaces of GENEOs, provided that the spaces of data are compact. Incidentally, we stress that the previously stated compactness does not hold for expansive equivariant operators, so justifying our interest in GENEOs.

The development of a good topological and geometric theory of the spaces of GENEOs could indeed produce new methods for approximating external agents in such spaces, suggest how to change such operators without losing their equivariance, benefit from their lattice structure with respect to the operations of maximization and minimization~\cite{FrQu17}, and allow to manage relations and conflicts that can arise in intelligent structures~\cite{Frosini2009DoesII}, just to make a few examples.
As for the link between GENEOs and Topological Data Analysis (with particular reference to \emph{persistent homology}), we refer the interested reader to \cite{FrJa16,Bergomi2019}.
A central role in this link is taken by the so-called \emph{natural pseudo-distance}~\cite{DoFr04,DoFr07,DoFr09,FrMu99}.

In summary, the use of GENEOs in machine learning requires to explore the global structure of the spaces these operators belong to, in order to fully benefit from their potential. In particular, this paper is devoted to study the global structure of the space of all linear $G$-equivariant non-expansive operators transitively acting on a finite signal domain $X$, with respect to the concept of \emph{permutant measure}.
}

We already know that every linear $G$-equivariant operator (GEO) can be represented as a $G$-con\-vo\-lu\-tion, provided that the group $G$ is compact and its action is transitive~\cite{pmlr-v80-kondor18a}. Unfortunately, the computation of the integral representing such a convolution is usually not trivial since in many applications the group $G$ is far from being small. In order to solve this problem, a new technique to build GEOs with respect to a group $G$ has been recently proposed, based on the concept of \emph{permutant}~\cite{CaFrQu18,CoFrQu22}. In plain words, a permutant is a collection of automorphisms of the domain $X$ of the signals we are interested in, under the assumption that such a collection is stable for the conjugation action of the group $G$.
A permutant is not required to be a group.
{When a non-empty permutant $H$ for $G$ is available, we can define a GEO $F$ for $G$ by setting
$F(\varphi):=\sum_{h\in H}\varphi h^{-1}$ for every admissible signal $\varphi$. The main benefit of this procedure is that the permutant $H$ can be much smaller than the group $G$, so that the computation of the GEO defined by $H$ can be much simpler than the computation of GEOs represented as $G$-convolutions.

In this paper we extend the definitions of permutant and GEO associated with a permutant by introducing the definitions of \emph{permutant measure} and \emph{GEO associated with a permutant measure}.
If $\mu$ is a permutant measure on the group $\mathrm{Aut}(X)$ of all permutations on $X$, we can define a GEO $F_\mu$ for $G$ by setting $F_\mu(\varphi):=\sum_{h\in \mathrm{Aut}(X)}\varphi h^{-1}\mu(h)$ for every admissible signal $\varphi$.}
In our mathematical setting we can study the relationship between linear GEOs and GEOs associated with permutant measures. In particular, we can prove that these two concepts coincide, provided that the group $G$ transitively acts on a finite signal domain $X$.
This theorem is based on the Birkhoff–von Neumann decomposition of doubly stochastic matrices and some well known facts in group theory. Its statement makes available a new method to build all linear GEOs with respect to a transitively acting group in the finite setting.
As a final step, we get our main result by adapting the previous theorem to the case of GENEOs, by taking into account the non-expansiveness condition.

{We stress that this paper is not focused on direct applications of GENEOs, but on a piece of research functional to the development of a good general theory of these operators. Our long-term goal is the one of making available equivariant operators that are both easily computable and predictable in their behavior, so allowing for their use in geometric deep learning.
}

The outline of the paper is as follows. In Section~\ref{sec:math-setting} we recall the main definitions in our mathematical setting, and introduce the concept of permutant measure together with some of its properties. In Section~\ref{main_result} we prove our main result about GEOs (Theorem~\ref{mainresult}), which is adapted to GENEOs in Section~\ref{main_result2} (Theorem~\ref{mainthm2}).
{In Section~\ref{example} we illustrate a toy example, showing possible advantages of the use of GENEOs built by permutant measures.}
In Section~\ref{discussion} we conclude the paper with a brief discussion.


%

\section{Mathematical setting}\label{sec:math-setting}

Let $\R^X\cong \R^n$ be the vector space of all functions from a finite set $X=\{x_1,\ldots,x_n\}$ to $\R$. We would like to recall that $\mathbb{R}^X$ has the canonical basis $\{{\mathds{1}}_{x_j}\}_{j}$, where ${\mathds{1}}_{x}\colon X \to \mathbb{R}$ is the function taking the value $1$ at $x$ and the value $0$ at every point $y$ with $y\neq x$.
We also consider the group $\mathrm{Aut}(X)$ of all permutations on $X$ and a subgroup $G$ of $\mathrm{Aut}(X)$. $\mathrm{Aut}(X)$ and $G$ naturally act on $\R^X$ by composition on the right. We endow $\R^X$ with the $L^\infty$-norm: $\|\varphi\|_\infty:=\max\{|\varphi(x_i)|:1\le i\le n\}$.

\begin{rem}\label{remPhi}
If we endow $X$ with the discrete topology, $\R^X$ coincides with $C^0(X,\R)$.
\end{rem}

In this paper we will use the multiplicative notation to denote the composition of functions, and the cycle notation to represent permutations.

\subsection{Group equivariant operators}\label{GEO}

We give the following definition.

\begin{defn}
A \emph{Group Equivariant Operator} (GEO) for $(\R^X,G)$ (with respect to the identity $\id_G \colon g \mapsto g$) is a function
$F \colon {\R^X} \longrightarrow {\R^X}$ such that $F(\varphi g)=F(\varphi) g$,
for all $\varphi\in {\R^X}$ and $g \in G$.
\end{defn}

An important subset of the set of GEOs is given by the set of \emph{Group Equivariant Non-Expansive Operators} (GENEOs),
i.e., GEOs $F$ such that $\| F(\varphi_{1})-F(\varphi_{2})\|_{\infty} \le \| \varphi_{1} -\varphi_{2}\|_{\infty}$,
for all $\varphi_{1},\varphi_{2}\in \R^X$.
In a more general framework, GEOs and GENEOs can be defined from $(\R^X,G)$ to $(\R^Y,H)$ with respect to a group homomorphism $T \colon G \to H$, where equivariance means that $F(\varphi g)=F(\varphi) T(g)$,
for all $\varphi\in {\R^X}$ and $g \in G$.
For further information, we refer the reader to~\cite{Bergomi2019}.
Obviously, the set of GEOs from $(\R^X,G)$ to $(\R^X,G)$ is not empty because it contains at least the identity operator $\id_{\R^X} \colon \varphi \mapsto \varphi$.


\subsection{Permutants and permutant measures}\label{permutants}

\begin{defn}
A finite signed measure ${\mu}$ on $\mathrm{Aut}(X)$ is called a \emph{permutant measure} with respect to $G$ if each subset $H$ of $\mathrm{Aut}(X)$ is measurable and $\mu$ is invariant under the conjugation action of $G$ (i.e., ${\mu}(H)={\mu}(g H g^{-1})$ for every $g\in G$). Equivalently, we can say that a signed measure ${\mu}$ on $\mathrm{Aut}(X)$ is a permutant measure with respect to $G$ if each singleton $\{h\}\subseteq \mathrm{Aut}(X)$ is measurable and ${\mu}(\{h\})={\mu}(\{g h g^{-1}\})$ for every $g\in G$.
\end{defn}

With a slight abuse of notation, we will denote by ${\mu}(h)$ the signed measure of the singleton $\{h\}$ for each $h\in \mathrm{Aut}(X)$.

\begin{ex}\label{exS1}
Let us consider a positive integer number $n$ and the finite set $$X:=\left\{\left(\cos \frac{2\pi k}{n},\sin \frac{2\pi k}{n}\right)\in\R^2:k\in \N, 0\le k\le n-1\right\}.$$ Let $G$ be the group of all rotations of $X$ of an angle $\alpha$ around the point $(0,0)$, with $\alpha$ multiple of $\frac{2\pi}{n}$. After fixing an integer number $m$, let us consider the map $\bar h\in \mathrm{Aut}(X)$ that takes each point $\left(\cos \frac{2\pi k}{n},\sin \frac{2\pi k}{n}\right)$ to
the point $\left(\cos \frac{2\pi (k+m)}{n},\sin \frac{2\pi (k+m)}{n}\right)$. Moreover, we define
the function $\mu_1\colon \mathcal{P}(\mathrm{Aut}(X))\to \R$ that takes each subset $C$ of $\mathrm{Aut}(X)$ to $1$ if $\bar h\in C$ and to $0$ if $\bar h\notin C$, where $\mathcal{P}(\mathrm{Aut}(X))$ is the power set of $\mathrm{Aut}(X)$.
Since the orbit of $\bar h$ under the conjugation action of $G$ is the singleton $\{\bar h\}$, the function $\mu_1$ is a permutant measure. We also observe that while the cardinality of $G$ is $n$, the cardinality of the support $\mathrm{supp}(\mu_1):=\{h\in \mathrm{Aut}(X):\mu_1(h)\neq 0\}$ of the signed measure  $\mu_1$ is $1$.
\end{ex}

\begin{ex}\label{exS2}
Let us consider the set $X$ of the vertices of a cube in $\R^3$, and the group $G$ of the orientation-preserving isometries of $\R^3$ that take
$X$ to $X$. Let $\pi_1,\pi_2,\pi_3$ be the three planes that contain the center of mass of $X$ and are parallel to a face of the cube. Let $h_i:X\to X$ be the orthogonal symmetry with respect to $\pi_i$, for $i\in \{1,2,3\}$. We have that the set $\{h_1,h_2,h_3\}$ is an orbit under the conjugation action of $G$. We can now define a permutant measure $\mu_2$ on the group $\mathrm{Aut}(X)$ by setting $\mu_2(h_1)=\mu_2(h_2)=\mu_2(h_3)= c$, where $c$ is a positive real number, and $\mu_2(h)=0$ for any $h\in \mathrm{Aut}(X)$ with $h\notin \{h_1,h_2,h_3\}$.
We also observe that while the cardinality of $G$ is $24$, the cardinality of the support $\mathrm{supp}(\mu_2):=\{h\in \mathrm{Aut}(X):\mu_2(h)\neq 0\}$ of the signed measure $\mu_2$ is $3$.
\end{ex}

Permutant measures give a simple method to build GEOs, as shown by the following result.

\begin{prop}\label{mainprop}
If $\mu$ is a permutant measure with respect to $G$, then the map $F_\mu:\R^X\to\R^X$ defined by setting
$F_\mu(\varphi):=\sum_{h\in {\mathrm{Aut}(X)}}\varphi h^{-1}\ {\mu}(h)$ is a linear GEO.
\end{prop}

\begin{proof}
Since $\mathrm{Aut}(X)$ linearly acts on $\R^X$ by composition on the right, $F_\mu$ is linear. Moreover, for every $\varphi\in\R^X$ and every $g\in G$
\begin{align}\label{f1}
F_\mu(\varphi g)&=\sum_{h\in {\mathrm{Aut}(X)}}\varphi g h^{-1}\ {\mu}(h)\\
&=\sum_{h\in {\mathrm{Aut}(X)}}\varphi g h^{-1}g^{-1}g\ {\mu}(g h g^{-1})\nonumber\\
&=\sum_{f\in {\mathrm{Aut}(X)}}\varphi f^{-1} g\ {\mu}(f)\nonumber\\
&=F_\mu(\varphi) g, \nonumber
\end{align}
since $\mu(h)={\mu}(g h g^{-1})$ and the map $h\mapsto f:=g h g^{-1}$ is a bijection from ${\mathrm{Aut}(X)}$ to ${\mathrm{Aut}(X)}$.
\end{proof}

Obviously, $F_\mu(\varphi)=\sum_{h\in {\mathrm{Aut}(X)}}\varphi h^{-1}\ {\mu}(h)=
\sum_{h\in {\mathrm{supp}(\mu)}}\varphi h^{-1}\ {\mu}(h)$, where $\mathrm{supp}(\mu):=\{h\in \mathrm{Aut}(X):\mu(h)\neq 0\}$. In Examples~\ref{exS1} and \ref{exS2} $|\mathrm{supp}(\mu_i)|\ll |G|$ for $i=1,2$, and hence in those cases summations on $\mathrm{supp}(\mu_i)$ are simpler than summations on the group $G$. The condition $|\mathrm{supp}(\mu)|\ll |G|$ is not rare in examples and is the main reason to build GEOs by means of permutant measures, instead of using the representation of GEOs as $G$-convolutions.

\begin{ex}\label{exS1andS2}
The GEOs associated with the permutant measures defined in Examples~\ref{exS1} and \ref{exS2} are respectively $F_{\mu_1}(\varphi)=\varphi \bar h^{-1}$ and
$F_{\mu_2}(\varphi)=c\varphi h_1^{-1}+c\varphi h_2^{-1}+c\varphi h_3^{-1}$.
\end{ex}

It is interesting to observe that the set $\PMe(G)$ of permutant measures with respect to $G$ is a lattice.
Indeed, if $\mu_1,\mu_2\in \PMe(G)$, then the measures $\mu',\mu''$ on $\mathrm{Aut}(X)$, respectively defined by setting $\mu'(h):=\min\{\mu_1(h),\mu_2(h)\}$ and $\mu''(h):=\max\{\mu_1(h),\mu_2(h)\}$, still belong to
$\PMe(G)$. Moreover, if $\mu\in \PMe(G)$ then $|\mu|\in \PMe(G)$.
Furthermore, $\PMe(G)$ is closed under linear combination.
Therefore, $ \PMe(G)$ has a natural structure of real vector space.
We can compute the dimension of $ \PMe(G)$ by considering the conjugation action of $G$ on $\mathrm{Aut}(X)$.


\begin{prop}\label{dimP}
$\dim \PMe(G)=\left|\sfrac{\mathrm{Aut}(X)}{G}\right|$.
\end{prop}

\begin{proof}
Consider a permutant measure $\mu$ on $\mathrm{Aut}(X)$.  We define the function $f_\mu \colon \sfrac{\mathrm{Aut}(X)}{G} \to \mathbb{R}$ by setting $f_\mu(\mathcal{O})=\mu(h)$, where $h \in \mathcal{O}$. Since $\mu$ is invariant under the conjugation action of $G$, $f_\mu$ is well defined. One could easily check that
the map $\mu \mapsto f_\mu$ is an isomorphism between $\PMe(G)$ and the space $\mathbb{R}^{\sfrac{\mathrm{Aut}(X)}{G}}$ of all real-valued functions on $\sfrac{\mathrm{Aut}(X)}{G}$. Since $\{\mathds{1}_\mathcal{O}\}_{\mathcal{O}\in \sfrac{\mathrm{Aut}(X)}{G}}$ is a basis for $\mathbb{R}^{\sfrac{\mathrm{Aut}(X)}{G}}$, $ \dim \mathbb{R}^{\sfrac{\mathrm{Aut}(X)}{G}}= \left|\sfrac{\mathrm{Aut}(X)}{G}\right|$. Hence, the statement is proved.
\end{proof}

Proposition~\ref{dimP} and the well-known Burnside's Lemma imply that $\dim \PMe(G)=\frac{1}{|G|}\sum_{g\in G}|\mathrm{Aut}(X)^g|$. We recall that $\mathrm{Aut}(X)^g$ denotes the set of elements fixed by the action of $g$, i.e., $\mathrm{Aut}(X)^g:= \{ h \in \mathrm{Aut}(X) | g h g^{-1}=h\}$.

Let us now recall the concept of permutant~\cite{CaFrQu18,CoFrQu22}, which is related to the one of permutant measure.

\begin{defn}\label{defpermutant}
We say that a subset $H\subseteq \mathrm{Aut}(X)$ is a \emph{permutant} for $G$ if either $H=\emptyset$ or $gHg^{-1}=H$ for every $g\in G$.
\end{defn}

Note that a subset $H$ of $\mathrm{Aut}(X)$ is a permutant for $G$ if and only if $H$ is a union of orbits
for the conjugation action of $G$ on $\mathrm{Aut}(X)$. It follows that the number of permutants for $G$ is equal to $2^{\left|\frac{\mathrm{Aut}(X)}{G}\right|}$.
Let us denote by $\Perm(G)$ the set of all permutants for $G$.
From Proposition~\ref{dimP} the next corollary immediately follows.

\begin{cor}\label{cordimPM}
$\dim \PMe(G)=\log_2 |\Perm(G)|$.
\end{cor}

The following definition extends the one of versatile group (cf.~\cite{CaFrQu18}) and is of use in studying permutants.
\begin{defn}\label{defkwv}
If $k$ is a positive integer, we say that the group $G\subseteq \mathrm{Aut}(X)$ is \emph{$k$-weakly versatile} if for every pair $(x,z)\in X\times X$ with $x\neq z$ and every subset $S$ of $X$  with $|S|\le k$, a $g\in G$ exists such that $g(x)=x$ and $g(z)\notin S$.
\end{defn}

The previous definition allows us to highlight an interesting property of permutants.

\begin{lem}\label{numperm}
If $G$ is $k$-weakly versatile, then every permutant $H\neq \emptyset,\{\id_X\}$ has cardinality strictly greater than $k$.
\end{lem}

\begin{proof}
By contradiction, let us assume that a non-empty permutant $H=\{h_1,\ldots,h_r\}\neq \{\id_X\}$ exists,
with $1\le r\le k$.
Since $H\neq \{\id_X\}$, we can assume that $h_1$ is not the identity. Let us take a point $x\in X$ such that $h_1(x)\neq x$ and set $z:=h_1(x)$, $S:=\{h_1(x),\ldots,h_r(x)\}$. $G$ is $k$-weakly versatile and hence a $g\in G$ exists, such that $g(x)=x$ and $g(z)\notin S$. Since $gHg^{-1}=H$, an index $i$ exists such that
$gh_1=h_ig$. It follows that $g(z)=g(h_1(x))=h_i(g(x))=h_i(x)\in S$, against the assumption that $g(z)\notin S$.
\end{proof}

 We stress that when the group $G$ becomes larger and larger the lattice $\PMe(G)$ becomes smaller and smaller. This duality implies that the method described by Proposition~\ref{mainprop} is particularly interesting when $G$ is large. In some sense, this duality is analogous  to the one described in~\cite[Subsection 3.1]{FrJa16}.

%


\section{Representation of linear GEOs via permutant measures}\label{main_result}

A natural question arises from Proposition~\ref{mainprop}: Which linear GEOs can be represented as GEOs associated with a permutant measure?

We can prove the following result.

\begin{thm}\label{mainthm}
If $G$ transitively acts on $X$, then for every linear group equivariant operator $F$ for $(\R^X,G)$ a permutant measure ${\mu}$ exists such that
$$F(\varphi)=F_\mu(\varphi):=\sum_{h\in {\mathrm{Aut}(X)}}\varphi h^{-1}\ {\mu}(h)$$ for every $\varphi\in {\R^X}$, and
$\sum_{h\in\mathrm{Aut}(X)}{|\mu(h)|}=
\max_{\varphi\in\R^X\setminus\{\mathbf{0}\}}\frac{\|F(\varphi)\|_\infty}{\|\varphi\|_\infty}$.
\end{thm}

In order to prove this statement, let us consider
the matrix $B=(b_{ij})$ associated with $F$ with respect to the basis
$\{{\mathds{1}}_{x_1},\ldots,{\mathds{1}}_{x_n}\}$.

\begin{rem}
We observe that ${\mathds{1}}_{x} {h}^{-1}={\mathds{1}}_{{h}(x)}$ for every $h\in \mathrm{Aut}(X)$ and every $x\in X$.
\end{rem}

In the following, for every $g\in G$ we will denote by $\sigma_g:\{1,\ldots,n\}\to \{1,\ldots,n\}$ the function defined by setting
$\sigma_g(j)=i$ if and only if $g(x_j)=x_i$. We observe that $\sigma_{g^{-1}}=\sigma_g^{-1}$.

We need the following lemmas.

\begin{lem}\label{lempermutation}
An $n$-tuple of real numbers $\alpha=(\alpha_1,\ldots,\alpha_{n})$ exists such that each row and each column of ${B}$ can be obtained by permuting $\alpha$.
\end{lem}

\begin{proof}
Let us choose a function ${\mathds{1}}_{x_j}$ and a permutation $g \in G$.
By equivariance we have that
$${F}({\mathds{1}}_{x_j}g )= {F}({\mathds{1}}_{x_j})g.$$
The left-hand side of the equation can be rewritten as:
$${F}({\mathds{1}}_{x_j}g )= {F}({\mathds{1}}_{g^{-1}(x_j)} )= \sum_{i=1}^{n}b_{i\sigma_g^{-1}(j)}{\mathds{1}}_{x_i}.$$
	
On the right-hand side we get
$${F}({\mathds{1}}_{x_j})g=\left(\sum_{i=1}^{n}b_{ij}{\mathds{1}}_{x_i}\right) g= \sum_{i=1}^{n}b_{ij}({\mathds{1}}_{x_i}g)= \sum_{i=1}^{n}b_{ij}({\mathds{1}}_{g^{-1}(x_i)})=\sum_{s=1}^{n}b_{\sigma_g(s)j}{\mathds{1}}_{x_s}$$
by setting $x_s=g^{-1}(x_i)$.
Therefore, we obtain the following equation:
$$\sum_{i=1}^{n}b_{i\sigma_g^{-1}(j)}{\mathds{1}}_{x_i}=\sum_{s=1}^{n}b_{\sigma_g(s)j}{\mathds{1}}_{x_s}.$$
This immediately implies that $b_{i\sigma_g^{-1}(j)}=b_{\sigma_g(i)j}$, for any $i \in \{1, \dots, n\}$.
Since this equality holds for
any $j \in \{1, \dots, n\}$ and any $g \in G$, we have that $b_{ij}=b_{\sigma_g(i)\sigma_g(j)}$ for every $i,j \in \{1, \dots, n\}$
and every $g\in G$.
	
Now we are ready to show that all the rows of ${B}$ are permutations of the first row, and all the columns are permutations of the first column.
Since $G$ is transitive, for every $ p,q\in\{1,\dots,n\}$ there exists
$g_{pq}\in G$ such that $g_{pq}(x_p)=x_q$.
Consider the $\bar \imath$-th row of $B$. We know that $b_{ \bar\imath j}
=b_{\sigma_{g_{\bar\imath 1}}(\bar \imath) \sigma_{g_{\bar\imath 1}}(j)}
=b_{ 1\sigma_{g_{\bar\imath 1}}(j)}$, for any $j \in \{1, \dots , n\}$.
Since $\sigma_{g_{\bar\imath 1}}$ is a permutation, the $\bar\imath$-th row is a permutation of the first row.
By the same arguments, we can assert that every column of $B$ is a permutation of the
first column of $B$.
	
Let us now consider a real number $y$, and denote by $r(y)$ (respectively $s(y)$) the number of times $y$ occurs in each row (respectively column) of ${B}$.
Both $nr(y)$ and $ns(y)$ represent the number of times $y$ appears in $B$.
Since $nr(y)=ns(y)$, each row and column contains the same elements (counted with multiplicity). Hence, the statement of our lemma is proved.
\end{proof}

The following result is well known~\cite{DKPU18}.

\begin{lem}[{Birkhoff–von Neumann decomposition}]\label{lemPk}
Let $M$ be a $n\times n$ real matrix with non-negative entries, such that both the sum of the elements of each row and the sum of the elements of each column is equal to $\bar c$. Then for every $h\in {\mathrm{Aut}(X)}$ a non-negative real number $c(h)$ exists such that
$\sum_{h\in {\mathrm{Aut}(X)}}c(h) = \bar c$ and $M=\sum_{h\in {\mathrm{Aut}(X)}}c(h) P(h)$, where $P(h)$ is the permutation matrix associated with $h$.
\end{lem}

We recall that the permutation matrix associated with the permutation $h:X\to X$ is the $n\times n$ real matrix $({p_{ij}(h)})$ defined by setting ${p_{ij}(h)}=1$ if ${h}(x_j)=x_i$ and ${p_{ij}(h)}=0$ if ${h}(x_j)\neq x_i$. Equivalently, we can define the permutation matrix associated with the permutation $h:X\to X$ as the $n\times n$ real matrix $P(h)$ such that $P(h)e_j =e_{\sigma_h(j)}$ for every column vector
$e_j:=\prescript{t}{}(0,\ldots,1,\ldots,0)\in \R^n$ (where $1$ is in the $j$-th position).

We observe that $P\left(h^{-1}\right)=P(h)^{-1}$ and $P(h_1h_2)=P(h_1)P(h_2)$ for every $h,h_1,h_2\in {\mathrm{Aut}(X)}$.

\begin{rem}
In general, the representation ${M}=\sum_{h\in {\mathrm{Aut}(X)}}c(h) P(h)$, stated in Lemma~\ref{lemPk}, is not unique. As an example, consider the set $X=\{1,2,3\}$ and the group $G = \mathrm{Aut}(X)$. Let ${F}\colon \mathbb{R}^X\to \mathbb{R}^X$ be the
linear application that maps ${\mathds{1}}_{j}$ to $\sum_{i\in X}\mathds{1}_i$, for any $j \in X$. One could easily check that ${F}$ is
a linear GEO for $(\R^X,G)$. Indeed, we have that ${F}({\mathds{1}}_{j} h)={F}({\mathds{1}}_{j})={F}({\mathds{1}}_{j}) h$ for any $j \in X$ and any $h \in \mathrm{Aut}(X)$. The matrix ${B}$ associated with ${F}$ with respect to the basis $\{{\mathds{1}}_{j}\}_j$ is:
	$${B}=\begin{pmatrix}
	1 & 1 & 1 \\
	1 & 1 & 1 \\
	1 & 1 & 1
	\end{pmatrix}.$$
	One could represent ${B}$ at least in two different ways:
	$${B}=
	\begin{pmatrix}
	1 & 0 & 0 \\
	0 & 1 & 0 \\
	0 & 0 & 1
	\end{pmatrix}+
	\begin{pmatrix}
	0 & 1 & 0 \\
	0 & 0 & 1 \\
	1 & 0 & 0
	\end{pmatrix}+
	\begin{pmatrix}
	0 & 0 & 1 \\
	1 & 0 & 0 \\
	0 & 1 & 0
	\end{pmatrix} $$ and $${B}=
	\begin{pmatrix}
	0 & 0 & 1 \\
	0 & 1 & 0 \\
	1 & 0 & 0
	\end{pmatrix}+
	\begin{pmatrix}
	0 & 1 & 0 \\
	1 & 0 & 0 \\
	0 & 0 & 1
	\end{pmatrix}+
	\begin{pmatrix}
	1 & 0 & 0 \\
	0 & 0 & 1 \\
	0 & 1 & 0
	\end{pmatrix}.$$
\end{rem}

We proceed in our proof of Theorem~\ref{mainthm} by taking the linear maps ${F^\oplus},{F^\ominus}:\R^X\to\R^X$ defined by setting
${F^\oplus}({\mathds{1}}_{x_j}) := \sum_{i=1}^{n}\max \{b_{ij},0\} {\mathds{1}}_{x_i}$
and
${F^\ominus}({\mathds{1}}_{x_j}) := \sum_{i=1}^{n}\max \{-b_{ij},0\} {\mathds{1}}_{x_i}$
for every index $j\in \{1,\ldots,n\}$. We can easily check that
\begin{enumerate}
  \item ${F^\oplus},{F^\ominus}$ are linear GEOs;
  \item The matrices associated with
${F^\oplus}$
and
${F^\ominus}$ with respect to the basis $\{{\mathds{1}}_{x_1},\ldots,{\mathds{1}}_{x_n}\}$ of $\mathbb{R}^X$ are
${B^\oplus}=\left(b^\oplus_{ij}\right)=\left(\max\{b_{ij},0\}\right)$
and
${B^\ominus}=\left(b^\ominus_{ij}\right)=\left(\max\{-b_{ij},0\}\right)$, respectively (in particular, ${B^\oplus},{B^\ominus}$ are non-negative matrices);
  \item $F={F^\oplus}-{F^\ominus}$ and $B={B^\oplus}-{B^\ominus}$;
  \item \label{property4} Lemma~\ref{lempermutation} and the definitions of ${B^\oplus}$, ${B^\ominus}$ imply that
two $n$-tuples of real numbers
$\alpha^\oplus=(\alpha^\oplus_1,\ldots,\alpha^\oplus_{n})$,
$\alpha^\ominus=(\alpha^\ominus_1,\ldots,\alpha^\ominus_{n})$ exist such that each row and each column of ${B^\oplus}$ can be obtained by permuting $\alpha^\oplus$, and each row and each column of ${B^\ominus}$ can be obtained by permuting $\alpha^\ominus$.
\end{enumerate}


From Property (\ref{property4}) and Lemma~\ref{lemPk} this result follows:

\begin{cor}\label{cor1}
For every $h\in {\mathrm{Aut}(X)}$ two non-negative real numbers $c^\oplus(h),c^\ominus(h)$ exist, such that ${F^\oplus}(\varphi)=\sum_{h\in {\mathrm{Aut}(X)}}{c^\oplus(h)}\varphi {h}^{-1}$ and
${F^\ominus}(\varphi)=\sum_{h\in {\mathrm{Aut}(X)}}{c^\ominus(h)}\varphi {h}^{-1}$ for every $\varphi\in {\R^X}$.
\end{cor}

\begin{proof}
Let us start by considering the statement concerning $c^\oplus(h)$ and $F^\oplus(h)$.
Without loss of generality, since ${F^\oplus}$ is linear, it will suffice to prove
the existence of a suitable non-negative function ${c^\oplus(h)}$, such
that  ${F^\oplus}({\mathds{1}}_{x_j})=\sum_{h\in {\mathrm{Aut}(X)}}{c^\oplus(h)}{\mathds{1}}_{x_j} {h}^{-1}$, for any $j\in \{1, \dots , n\}$.
The column coordinate vector of the function ${F^\oplus}({\mathds{1}}_{x_j})$ relative to
the basis $\{{\mathds{1}}_{x_1},\ldots,{\mathds{1}}_{x_n}\}$
is ${B^\oplus} e_j$.
Property (\ref{property4}) and Lemma~\ref{lemPk} imply that for every $h\in {\mathrm{Aut}(X)}$ a non-negative real number ${c^\oplus(h)}$
exists, such that
$${B^\oplus} e_j = \sum_{h\in {\mathrm{Aut}(X)}}{c^\oplus(h)} P(h) e_j=\sum_{h\in {\mathrm{Aut}(X)}}{c^\oplus(h)} e_{\sigma_h(j)}.$$
Since the column vector $e_{\sigma_h(j)}$ represents the column coordinate vector of the function ${\mathds{1}}_{h(x_j)}$ relative to the basis
$\{{\mathds{1}}_{x_1},\ldots,{\mathds{1}}_{x_n}\}$,
we can conclude that $${F^\oplus}({\mathds{1}}_{x_j})=\sum_{h\in {\mathrm{Aut}(X)}}{c^\oplus(h)} {\mathds{1}}_{h(x_j)}
=\sum_{h\in {\mathrm{Aut}(X)}}{c^\oplus(h)} {\mathds{1}}_{x_j}h^{-1}.$$
The proof of the statement concerning $c^\ominus$ and $F^\ominus$ is analogous.
\end{proof}

\begin{rem}
In general, the function $c \colon \mathrm{Aut}(X) \to \R$
associated with the Birkhoff–von Neumann decomposition
does not induce a permutant measure,
i.e., the function $\mu_c$ that takes each subset $H$ of $\mathrm{Aut}(X)$ to the value $\mu_c(H):=\sum_{h\in H}c(h)$ is not a permutant measure.
For example, let us consider the set $X=\{1,2,3,4\}$ and the group $S_4$ of all
permutations of $X$.
Let us define a linear GEO $F: \mathbb{R}^X \to \mathbb{R}^X$
for $(\R^X,S_4)$
by setting $F(\mathds{1}_j)=\sum_{i\in X}\mathds{1}_i$, for every index $j$. After fixing the basis $\{\mathds{1}_j\}_j$, the matrix $B$ associated with $F$ has the following form:
	$$B=\begin{pmatrix}
	1 & 1 & 1 & 1\\
	1 & 1 & 1 & 1\\
	1 & 1 & 1 & 1\\
	1 & 1 & 1 & 1\\
	\end{pmatrix}.$$
	
	As guaranteed by Lemma~\ref{lemPk}, $B$ can be decomposed as follows:
	$$B= P(\id_X) + P(\sigma) + P(\sigma^2) + P(\sigma^3),$$
	where $\sigma= \cycle{1,2,3,4} \in S_4$, in cycle notation.
	Let $\langle\sigma \rangle$ be the cyclic group generated by $\sigma$.
	{\color{black} The function $c:\mathrm{Aut}(X)\to \R$ associated with the previous decomposition of $B$ is defined as follows: $c({h})=1$ if $h \in \langle\sigma \rangle$, otherwise $c({h})=0$. Let us now consider the permutation $g=\cycle{1,2}\in S_4$, in cycle notation.
Since $\sigma^2=\cycle{1,3}\cycle{2,4}$, we have that:}
	$$c(g \sigma^2 g^{-1})= c(\cycle{1,2}\cycle{1,3}\cycle{2,4}\cycle{1,2})=c(\cycle{1,4}\cycle{2,3})=0.$$
	Since $c(\sigma^2)=1$, $c$ is not invariant under the conjugation action of $S_4$, and hence $\mu_c$ is not a permutant measure.
\end{rem}

Let us now go back to the proof of Theorem~\ref{mainthm} and consider the functions
${c^\oplus},{c^\ominus}:{\mathrm{Aut}(X)}\to \R$ introduced in Corollary~\ref{cor1}.
In order to define the permutant measure ${\mu}$ on ${\mathrm{Aut}(X)}$ we will need the next lemma.

\begin{lem}\label{lemcommute}
If $g\in G$ then ${B^\oplus}P(g)=P(g){B^\oplus}$.
\end{lem}

\begin{proof}
Let us consider a permutation $g \in G$. The function $R_{g^{-1}}\colon \mathbb{R}^X\to \mathbb{R}^X$, which maps $\varphi$ to $\varphi {g^{-1}}$, is a linear application. Furthermore, $R_{g^{-1}}({\mathds{1}}_{x_j})={\mathds{1}}_{x_j} {g^{-1}}={\mathds{1}}_{g(x_j)}$
for every index $j$.
Hence, the matrix $N$ associated to $R_{g^{-1}}$ with respect to the basis
$\{{\mathds{1}}_{x_1},\ldots,{\mathds{1}}_{x_n}\}$
verifies the equality $Ne_j=e_{\sigma_{g}(j)}$, so that $N=P(g)$
(we set
$e_j:=\prescript{t}{}(0,\ldots,1,\ldots,0)\in \R^n$, where $1$ is in the $j$-th position).
Since ${F^\oplus}$ is a GEO, the equality ${F^\oplus}R_{g^{-1}}=R_{g^{-1}}{F^\oplus}$ holds.
This immediately implies that ${B^\oplus}P(g)=P(g){B^\oplus}$.
\end{proof}

An analogous lemma holds for the matrix ${B^\ominus}$.

Lemma~\ref{lemcommute} guarantees that $P(g){B^\oplus}P(g)^{-1}={B^\oplus}$ for every $g\in G$. From this equality and Lemma~\ref{lemPk} it follows that
\begin{align}\label{f2}
{B^\oplus}&=\overbrace{\frac{1}{|G|}{B^\oplus}+\ldots +\frac{1}{|G|}{B^\oplus}}^\text{$|G|$ summands}\\
&=\frac{1}{|G|}\sum_{g\in G} P(g){B^\oplus}P(g)^{-1}\nonumber\\
&= \frac{1}{|G|}\sum_{g\in G} P(g)  \left(\sum_{{h\in {\mathrm{Aut}(X)}}}{c^\oplus(h)} P(h)\right) P(g)^{-1}\nonumber\\
&= \sum_{{h\in {\mathrm{Aut}(X)}}}\sum_{g\in G} \frac{{c^\oplus(h)}}{|G|} P(g)P(h)P(g)^{-1}\nonumber\\
&= \sum_{{h\in {\mathrm{Aut}(X)}}}\frac{{c^\oplus(h)}}{|G|}\sum_{g\in G} P(g h g^{-1})\nonumber.
\end{align}
Therefore, for every index $j$ we have that
\begin{align}\label{f3}
{B^\oplus}e_j &=\sum_{{h\in {\mathrm{Aut}(X)}}}\frac{{c^\oplus(h)}}{|G|}\sum_{g\in G} P(g h g^{-1})e_j\\
&=\sum_{{h\in {\mathrm{Aut}(X)}}}\frac{{c^\oplus(h)}}{|G|}\sum_{g\in G} e_{\sigma_{g h g^{-1}}(j)}\nonumber.
\end{align}
This means that
\begin{align}\label{f4}
{F^\oplus}({\mathds{1}}_{x_j})&=\sum_{{h\in {\mathrm{Aut}(X)}}}\frac{{c^\oplus(h)}}{|G|}\sum_{g\in G} {\mathds{1}}_{g h g^{-1}(x_j)}\\
&=\sum_{{h\in {\mathrm{Aut}(X)}}}\frac{{c^\oplus(h)}}{|G|}\sum_{g\in G} {\mathds{1}}_{x_j}g h^{-1} g^{-1}\nonumber.
\end{align}
Since ${F^\oplus}$ is linear, it follows that
\begin{align}\label{f5}
{F^\oplus}(\varphi)&=\sum_{{h\in {\mathrm{Aut}(X)}}} \frac{{c^\oplus(h)}}{|G|}\sum_{g\in G} \varphi g {h}^{-1} g^{-1}
\end{align}
for every $\varphi\in\R^X$.

We observe that the permutations $g {h}^{-1} g^{-1}$ in the previous summation are not guaranteed to be different from each other, for $g$ varying in $G$ and $h$ varying in ${\mathrm{Aut}(X)}$.

For each $h\in {\mathrm{Aut}(X)}$, let us consider the orbit $\mathcal{O}(h)$ of $h$ under the conjugation action of $G$ on ${\mathrm{Aut}(X)}$,
and set
\begin{align*}
{\mu^\oplus}(h)&:=\sum_{{f}\in \mathcal{O}(h)}\frac{c^\oplus({f})}{|\mathcal{O}({f})|}=\sum_{{f}\in \mathcal{O}(h)}\frac{c^\oplus({f})}{|\mathcal{O}(h)|}\\
{\mu^\ominus}(h)&:=\sum_{{f}\in \mathcal{O}(h)}\frac{c^\ominus({f})}{|\mathcal{O}({f})|}=\sum_{{f}\in \mathcal{O}(h)}\frac{c^\ominus({f})}{|\mathcal{O}(h)|}.
\end{align*}

In other words, we define the measures $\mu^\oplus(h),\mu^\ominus(h)$ of each permutation $h$ as the averages of the functions $c^\oplus,c^\ominus$ along the orbit of $h$ under the conjugation action of $G$. Let $G_h$ be the stabilizer subgroup of $G$ with respect to $h$, i.e., the subgroup of $G$ containing the elements that fix $h$ by conjugation. We recall that by conjugating $h$ with respect to every element of $G$ we obtain each element of the orbit $\mathcal{O}(h)$ exactly $|G_h|$ times,
and the well-known relation $|G_h||\mathcal{O}(h)|=|G|$ (cf.~\cite{aschbacher_2000}).
Let us now set $\delta({f},h)=1$ if ${f}$ and $h$ belong to the same orbit under the conjugation action of $G$, and $\delta({f},h)=0$ otherwise.

We observe that the following properties hold for $f,h\in {\mathrm{Aut}(X)}$:
\begin{enumerate}
  \item $G_{h^{-1}}=G_{h}$;
  \item if ${f}\in \mathcal{O}(h)$ then $G_{{f}}$ is isomorphic to $G_{h}$;
  \item ${f}^{-1}\in \mathcal{O}(h^{-1})\iff {f}\in \mathcal{O}(h)\iff h\in \mathcal{O}({f})$.
\end{enumerate}
Therefore, equality (\ref{f5}) implies
\begin{align}\label{f6}
{F^\oplus}(\varphi )&=\sum_{{h\in {\mathrm{Aut}(X)}}}\frac{{c^\oplus(h)}}{|G|}|G_{h^{-1}}|\sum_{{f}^{-1}\in \mathcal{O}(h^{-1})}  \varphi  {{f}}^{-1}\\
&=\sum_{{h\in {\mathrm{Aut}(X)}}}\frac{{c^\oplus(h)}}{|G|}|G_{h}|\sum_{{f}\in \mathcal{O}(h)}  \varphi  {{f}}^{-1}\nonumber\\
&=\sum_{{h\in {\mathrm{Aut}(X)}}}\frac{{c^\oplus(h)}}{|G|}|G_{h}|\sum_{{f}\in {\mathrm{Aut}(X)}}  \delta({f},h) \varphi  {{f}}^{-1}\nonumber\\
&=\sum_{{{f}\in {\mathrm{Aut}(X)}}}\left(\sum_{h\in {\mathrm{Aut}(X)}}\frac{{c^\oplus(h)}}{|G|}|G_{h}|  \delta({f},h)\right) \varphi  {{f}}^{-1}\nonumber\\
&=\sum_{{{f}\in {\mathrm{Aut}(X)}}}\left(\sum_{h\in {\mathrm{Aut}(X)}}\frac{{c^\oplus(h)}}{|\mathcal{O}(h)|}  \delta({f},h)\right) \varphi  {{f}}^{-1}\nonumber\\
&=\sum_{{{f}\in {\mathrm{Aut}(X)}}}\left(\sum_{h\in \mathcal{O}({f})}\frac{{c^\oplus(h)}}{|\mathcal{O}(h)|} \right) \varphi  {{f}}^{-1}\nonumber\\
&=\sum_{{{f}\in {\mathrm{Aut}(X)}}}\varphi  {{f}}^{-1}\ {\mu^\oplus}({f}). \nonumber
\end{align}

The definition of ${\mu^\oplus}$ immediately implies that ${\mu^\oplus}(H)={\mu^\oplus}(g  H  g^{-1})$ for every $g\in G$ and every subset $H$ of ${\mathrm{Aut}(X)}$.
In other words, ${\mu^\oplus}$ is a non-negative permutant measure with respect to $G$.
Quite analogously, we can prove the equality ${F^\ominus}(\varphi )=\sum_{{{f}\in {\mathrm{Aut}(X)}}}\varphi  {{f}}^{-1}\ {\mu^\ominus}({f})$, and that ${\mu^\ominus}$ is a non-negative permutant measure with respect to $G$.
As a result, the function $\mu:={\mu^\oplus}-{\mu^\ominus}$ is a permutant measure and the equality ${F}(\varphi )=\sum_{{{f}\in {\mathrm{Aut}(X)}}}\varphi  {{f}}^{-1}\ {\mu}({f})$ holds, since $F={F^\oplus}-{F^\ominus}$.

It remains to prove that $\sum_{h\in\mathrm{Aut}(X)}{|\mu(h)|}
=
\max_{\varphi\in\R^X\setminus\{\mathbf{0}\}}\frac{\|F(\varphi)\|_\infty}{\|\varphi\|_\infty}$.

This statement is trivial if $F\equiv \mathbf{0}$,
since in this case $\mu$ is the null measure. Hence
we can assume that $F$ is not the null map and $B$ is not the null matrix.
In order to proceed, we need the next statement.

%
%
%

\begin{prop}\label{propc}
If $f_1,f_2\in\mathrm{Aut}(X)$ and an index $s\in \{1,\ldots,n\}$ exists, such that $f_1(x_s)=f_2(x_s)$ (i.e., $\sigma_{f_1}(s)=\sigma_{f_2}(s)$), then
either $c^\oplus(f_1)=0$, or $c^\ominus(f_2)=0$, or both.
\end{prop}

\begin{proof}
By applying the equality
$F^\oplus(\varphi)=\sum_{h\in {\mathrm{Aut}(X)}}c^\oplus(h)\varphi h^{-1}$ for $\varphi=\mathds{1}_{x_1}$,
we obtain that
\begin{align}\label{fc}
b^\oplus_{\sigma_{f_1}(s)s}&= \left(\sum_{i=1}^{n}b^\oplus_{is}{\mathds{1}}_{x_i}\right)(x_{\sigma_{f_1}(s)})\\
&= {F^\oplus}(\mathds{1}_{x_s})(x_{\sigma_{f_1}(s)})\nonumber\\
&= {F^\oplus}(\mathds{1}_{x_s})({f_1}(x_s))\nonumber\\
&=\sum_{{h\in {\mathrm{Aut}(X)}}}{c^\oplus(h)} \mathds{1}_{x_s}  {{h}}^{-1}({f_1}(x_s))\nonumber\\
&\ge {c^\oplus({f_1})} \mathds{1}_{x_s}  {{{f_1}}}^{-1}({f_1}(x_s))\nonumber\\
&= {c^\oplus({f_1})} \mathds{1}_{x_s} (x_s)\nonumber\\
&=c^\oplus({f_1}). \nonumber
\end{align}
Analogously, the inequality $b^\ominus_{\sigma_{f_2}(s)s} \ge c^\ominus({f_2})$ holds.
Therefore,

$c^\oplus({f_1})>0 \implies b^\oplus_{\sigma_{f_1}(s)s}>0 \implies b^\ominus_{\sigma_{f_1}(s)s}=0
\implies b^\ominus_{\sigma_{f_2}(s)s}=0
\implies c^\ominus({f_2})=0$.

It follows that
either $c^\oplus(f_1)=0$, or $c^\ominus(f_2)=0$, or both.
\end{proof}

\begin{cor}\label{corc}
For every $f\in\mathrm{Aut}(X)$ either $c^\oplus(f)=0$, or $c^\ominus(f)=0$, or both.
\end{cor}

\begin{proof}
Set $f_1=f_2$ in Proposition~\ref{propc}.
\end{proof}


Let us now set $c:=c^\oplus-c^\ominus$. Corollary~\ref{corc} implies that $|c(h)|=c^\oplus(h)+c^\ominus(h)$ for every $h\in\mathrm{Aut}(X)$.
The definitions of $\mu^\oplus$ and $\mu^\ominus$ immediately imply that
$\sum_{{f}\in \mathcal{O}(h)}{\mu^\oplus}(f)=\sum_{{f}\in \mathcal{O}(h)}c^\oplus({f})$ and
$\sum_{{f}\in \mathcal{O}(h)}{\mu^\ominus}(f)=\sum_{{f}\in \mathcal{O}(h)}c^\ominus({f})$
for each $h\in\mathrm{Aut}(X)$.
It follows that
$\sum_{{f}\in \mathcal{O}(h)}|{\mu}(f)|\le
\sum_{{f}\in \mathcal{O}(h)}{\mu^\oplus}(f)+\sum_{{f}\in \mathcal{O}(h)}{\mu^\ominus}(f)=
\sum_{{f}\in \mathcal{O}(h)}{c^\oplus}(f)+\sum_{{f}\in \mathcal{O}(h)}{c^\ominus}(f)=
\sum_{{f}\in \mathcal{O}(h)}|c({f})|$
for each $h\in\mathrm{Aut}(X)$, and hence $$\sum_{h\in\mathrm{Aut}(X)}|{\mu}(h)|\le\sum_{h\in\mathrm{Aut}(X)}|c({h})|.$$
By setting ${\mathds{1}}_{X}:=\sum_{j=1}^{n}{\mathds{1}}_{x_j}$ and
recalling Corollary~\ref{cor1}, we obtain $F^\oplus({\mathds{1}}_{X})=\left(\sum_{h\in {\mathrm{Aut}(X)}}{c^\oplus(h)}\right){\mathds{1}}_{X}$
and $F^\ominus({\mathds{1}}_{X})=\left(\sum_{h\in {\mathrm{Aut}(X)}}{c^\ominus(h)}\right){\mathds{1}}_{X}$.
Since any line in $B$ is a permutation of the first row of $B$, we get $F^\oplus({\mathds{1}}_{X})=\left(\sum_{j=1}^{n} b_{1j}^\oplus\right){\mathds{1}}_{X}$
and $F^\ominus({\mathds{1}}_{X})=\left(\sum_{j=1}^{n} b_{1j}^\ominus\right){\mathds{1}}_{X}$. As a consequence, the equalities
$\sum_{h\in\mathrm{Aut}(X)} c^\oplus({h})=\sum_{j=1}^{n} b_{1j}^\oplus$ and
$\sum_{h\in\mathrm{Aut}(X)} c^\ominus({h})=\sum_{j=1}^{n} b_{1j}^\ominus$ hold, and therefore
$\sum_{h\in\mathrm{Aut}(X)} |c({h})|=\sum_{h\in\mathrm{Aut}(X)} c^\oplus({h})+\sum_{h\in\mathrm{Aut}(X)} c^\ominus({h})=\sum_{j=1}^{n} b_{1j}^\oplus+\sum_{j=1}^{n} b_{1j}^\ominus=\sum_{j=1}^{n} |b_{1j}|$.

Let us now consider the function $\bar\varphi:=\sum_{j=1}^{n}\sgn(b_{1j}){\mathds{1}}_{x_j}\in\R^X\setminus\{\mathbf{0}\}$.
By recalling that any line in $B$ is a permutation of the first row of $B$, we have that
$\sum_{j=1}^{n}|b_{1j}|=|\sum_{j=1}^{n}b_{1j}\sgn\left(b_{1j}\right)|\ge |\sum_{j=1}^{n}b_{ij}\sgn\left(b_{1j}\right)|$ for every index $i$.
It follows that
$$\|F(\bar\varphi)\|_\infty=\left\|
\sum_{i=1}^{n}\left(\sum_{j=1}^{n}b_{ij}\sgn\left(b_{1j}\right)\right){\mathds{1}}_{x_i}
\right\|_\infty=\sum_{j=1}^{n}|b_{1j}|
= \sum_{h\in\mathrm{Aut}(X)}|c({h})|\ge
\sum_{h\in\mathrm{Aut}(X)}|{\mu}(h)|.$$

Since $\|\bar\varphi\|_\infty=1$,
$\frac{\|F(\bar\varphi)\|_\infty}{\|\bar\varphi\|_\infty}
\ge
\sum_{h\in\mathrm{Aut}(X)}|{\mu}(h)|$.

For every function $\varphi\in\R^X$ we have that
$F(\varphi)=F_\mu(\varphi):=\sum_{h\in {\mathrm{Aut}(X)}}\varphi h^{-1}\ {\mu}(h)$.
Hence,
$\|F(\varphi)\|_\infty\le\sum_{h\in {\mathrm{Aut}(X)}}\|\varphi h^{-1}\|_\infty\ |{\mu}(h)|
=\|\varphi\|_\infty\sum_{h\in {\mathrm{Aut}(X)}} |{\mu}(h)|$.
Therefore,
$\frac{\|F(\varphi)\|_\infty}{\|\varphi\|_\infty}\le\sum_{h\in\mathrm{Aut}(X)}{|\mu(h)|}$ for every $\varphi\in\R^X\setminus\{\mathbf{0}\}$.

In conclusion,
$\sum_{h\in\mathrm{Aut}(X)}{|\mu(h)|}=\max_{\varphi\in\R^X\setminus\{\mathbf{0}\}}\frac{\|F(\varphi)\|_\infty}{\|\varphi\|_\infty}$.
\qed


\begin{ex}\label{exsimplestcase}
The simplest non-trivial example concerning the statement of Theorem~\ref{mainthm} can be described as follows. Let $X=\{1,2\}$ and
$G={\mathrm{Aut}(X)}=\{\id_X,\cycle{1,2}\}$. Let us consider the linear GEO $F:\R^X\to\R^X$ defined by setting
$F({\mathds{1}}_{1}):={\mathds{1}}_{1} - {\mathds{1}}_{2}$ and $F({\mathds{1}}_{2}):={\mathds{1}}_{2}-{\mathds{1}}_{1}$.
By defining
$\mu(\id_X):=1$ and $\mu(\cycle{1,2}):=-1$, we get that $\mu$ is a permutant measure with respect to $G$ and ${F}(\varphi )=\sum_{{{h}\in {\mathrm{Aut}(X)}}}\varphi  {{h}}^{-1}\ {\mu}({h})$ for every $\varphi\in \R^X$.
Furthermore, $\sum_{h\in\mathrm{Aut}(X)}{|\mu(h)|}=2=\frac{\|F({\mathds{1}}_{1}-{\mathds{1}}_{2})\|_\infty}{\|{\mathds{1}}_{1}-{\mathds{1}}_{2}\|_\infty}
=\max_{\varphi\in\R^X\setminus\{\mathbf{0}\}}\frac{\|F(\varphi)\|_\infty}{\|\varphi\|_\infty}$.
\end{ex}

We now observe that the assumption that $G$ transitively acts on $X$ cannot be removed from Theorem~\ref{mainthm}.

\begin{ex}\label{ex1}
	Let us consider the set $X= \{1, 2\}$ and the group $G=\{\id_X\}\subseteq \mathrm{Aut}(X)=\{\id_X, \cycle{1,2}\}$.
Take the operator $F\colon {\R^X}\to{\R^X}$ defined by setting $F(\mathds{1}_i)=\mathds{1}_1$ for any $i \in X$.
Although
$F$ is a linear GEO, there does not exist a permutant measure $\mu$ on $\text{Aut}(X)$,
such that $F(\varphi)=F_\mu(\varphi):=\sum_{h\in {\mathrm{Aut}(X)}}\varphi h^{-1}\ {\mu}(h)$ for every $\varphi\in {\R^X}$.

By contradiction, let us assume that such a permutant measure $\mu$ exists.
Then,
$$ \mathds{1}_{1}=F(\mathds{1}_{1})= \mathds{1}_{1}\id_X\mu(\id_X) + \mathds{1}_1 \cycle{1,2}\mu(\cycle{1,2}) =\mathds{1}_{1}\mu(\id_X) +  \mathds{1}_2\mu(\cycle{1,2}).$$
	Since $\{\mathds{1}_1, \mathds{1}_2\}$
is a basis for $\R^X$,
the equalities $\mu(\id_X)=1$ and $\mu(\cycle{1,2})=0$ must hold.

It follows that
$$F(\mathds{1}_2)= \mathds{1}_{2}\id_X\mu(\id_X) + \mathds{2}\mathds{1}_1 \cycle{1,2}\mu(\cycle{1,2})  =\mathds{1}_{2}.$$
This contradicts the assumption that $F(\mathds{1}_2)=\mathds{1}_1$.
\end{ex}

\begin{ex}\label{exS4}
Let us set $X= \{1, 2, 3, 4\}$ and $G={\mathrm{Aut}(X)}$. Let $F$ be a linear GEO with respect to $G$.
Let $B=(b_{ij})$ be the matrix associated with $F$ with respect to the basis
$\{{\mathds{1}}_{x_1},\ldots,{\mathds{1}}_{x_n}\}$. In the proof of Lemma~\ref{lempermutation} we have seen that $b_{ij}=b_{\sigma_g(i)\sigma_g(j)}$ for any $g \in G$.
It follows that two values $\alpha,\beta\in\R$ exist, such that $b_{ij}=\alpha$ if $i=j$ and $b_{ij}=\beta$ if $i\neq j$.
By using the cycle notation, let us set $\sigma= \cycle{1,2,3,4} \in {\mathrm{Aut}(X)}$
and
$\langle\sigma \rangle=\{\id_X,\sigma,\sigma^2=\cycle{1,3}\cycle{2,4},\sigma^3=\cycle{1,4,3,2}\}$, i.e., the cyclic group generated by $\sigma$.
We have that
$B=\alpha P(\id_X)+\beta P(\sigma)+\beta P(\sigma^2)+\beta P(\sigma^3)$. Therefore, by setting
$c(\id_X):=\alpha$, $c(\sigma):=c(\sigma^2):=c(\sigma^3):=\beta$, and $c(h):=0$ for every $h\not\in \langle\sigma \rangle$, we get
$F(\varphi)=\sum_{h\in {\mathrm{Aut}(X)}} \varphi h^{-1}\ c(h)$.

However, the signed measure $c$ is not a permutant measure, since the orbits under the conjugation action of $G$ are the sets
$$\begin{array}{lcl}
\mathcal{O}(\id_X) &=& \{\id_X\} \\
\mathcal{O}({\sigma}) &=& \{\sigma= \cycle{1,2,3,4},\cycle{1,2,4,3},\cycle{1,3,2,4},\cycle{1,3,4,2},\cycle{1,4,2,3},\sigma^3=\cycle{1,4,3,2}\} \\
\mathcal{O}({\sigma^2}) &=& \{\cycle{1,2}\cycle{3,4},\sigma^2=\cycle{1,3}\cycle{2,4},\cycle{1,4}\cycle{2,3}\} \\
\mathcal{O}({\cycle{1,2}}) &=& \{\cycle{1,2},\cycle{1,3},\cycle{1,4},\cycle{2,3},\cycle{2,4},\cycle{3,4}\} \\
\mathcal{O}({\cycle{1,2,3}}) &=& \{\cycle{1,2,3},\cycle{1,2,4},\cycle{1,3,2},\cycle{1,3,4},\cycle{1,4,2},\cycle{1,4,3},\cycle{2,3,4},\cycle{2,4,3}\}
\end{array}$$
and, according to our definition, $c$ is not constant on the orbits $\mathcal{O}({\sigma})$ and $\mathcal{O}({\sigma^2})$.


Following the proof of Theorem~\ref{mainthm}, we can get a permutant measure $\mu$
by computing an average on the orbits. In other words, we can set
$$
\mu(h):=
\begin{cases}
c(\id_X)=\alpha, &\text{if } h=\id_X\\
\sum_{h\in \mathcal{O}(\sigma)} \frac{c(h)}{|O(\sigma)|}=\sum_{h\in \mathcal{O}(\sigma^2)} \frac{c(h)}{|O(\sigma^2)|}=\frac{\beta}{3}, &\text{if } h\in \mathcal{O}({\sigma})\cup \mathcal{O}({\sigma^2})\\
0, &\text{otherwise.}
\end{cases}
$$
By making this choice, the equality $F(\varphi)=
\sum_{g\in {\mathrm{Aut}(X)}}\varphi  h^{-1}\  \mu(h)$ holds for every $\varphi\in \R^X$, i.e., $F$ is the linear GEO associated with the permutant measure $\mu$.
\end{ex}

Proposition~\ref{mainprop} and Theorem~\ref{mainthm} immediately imply the following statement.

\begin{thm}\label{mainresult}
Assume that $G\subseteq \mathrm{Aut}(X)$ transitively acts on the finite set $X$ and $F$ is a map from $\R^X$ to $\R^X$. The map $F$ is a linear group equivariant operator for $(\R^X,G)$ if and only if a permutant measure ${\mu}$ exists such that
$F(\varphi)=\sum_{h\in {\mathrm{Aut}(X)}}\varphi h^{-1}\ {\mu}(h)$ for every $\varphi\in {\R^X}$.
\end{thm}

\section{Representation of linear GENEOs via permutant measures}\label{main_result2}

Our main result about the representation of linear GEOs can be adapted to GENEOs.

\begin{thm}\label{mainthm2}
Assume that $G\subseteq \mathrm{Aut}(X)$ transitively acts on the finite set $X$ and $F$ is a map from $\R^X$ to $\R^X$. The map $F$ is a linear group equivariant non-expansive operator for $(\R^X,G)$ if and only if a permutant measure ${\mu}$ exists such that
$F(\varphi)=\sum_{h\in {\mathrm{Aut}(X)}}\varphi h^{-1}\ {\mu}(h)$ for every $\varphi\in {\R^X}$,
and $\sum_{h\in {\mathrm{Aut}(X)}}|\mu(h)|\le 1$.
\end{thm}

\begin{proof}\label{proofmainthm2}
If $F$ is a linear group equivariant non-expansive operator for $(\R^X,G)$,
then Theorem~\ref{mainthm} guarantees that in $\PMe(G)$ a permutant measure ${\mu}$ exists, such that
$F(\varphi)=\sum_{h\in {\mathrm{Aut}(X)}}\varphi h^{-1}\ {\mu}(h)$ for every $\varphi\in {\R^X}$, and
$\sum_{h\in\mathrm{Aut}(X)}{|\mu(h)|}=\max_{\varphi\in\R^X\setminus\{\mathbf{0}\}}\frac{\|F(\varphi)\|_\infty}{\|\varphi\|_\infty}$.
Since $F$ is non-expansive, the inequality $\sum_{h\in {\mathrm{Aut}(X)}}|\mu(h)|\le 1$ follows.
This proves the first implication in our statement.

Let us now assume that a permutant measure ${\mu}$ exists such that
$$F(\varphi)=\sum_{h\in {\mathrm{Aut}(X)}}\varphi h^{-1}\ {\mu}(h)$$ for every $\varphi\in {\R^X}$,
with $\sum_{h\in {\mathrm{Aut}(X)}}|\mu(h)|\le 1$.
Then Proposition~\ref{mainprop} states that $F$ is a linear group equivariant operator for $(\R^X,G)$.
Moreover,
\begin{align}\label{fft}
\|F(\varphi)\|_\infty &=\left\|\sum_{h\in {\mathrm{Aut}(X)}}\varphi h^{-1}\ {\mu}(h)\right\|_\infty\nonumber\\
&\le\sum_{h\in {\mathrm{Aut}(X)}}\left\|\varphi h^{-1}\right\|_\infty|{\mu}(h)|\nonumber\\
&=\sum_{h\in {\mathrm{Aut}(X)}}\left\|\varphi \right\|_\infty|{\mu}(h)|\nonumber\\
&=\|\varphi\|_\infty\left(\sum_{h\in {\mathrm{Aut}(X)}} \ |{\mu}(h)|\right)\nonumber\\
&\le\|\varphi\|_\infty.\nonumber
\end{align}
This proves that $F$ is non-expansive, and concludes the proof of the second implication in our statement.
\end{proof}


\section{How GENEOs based on permutant measures could be used for transforming data}\label{example}

In this section we will illustrate an example of a possible application of GENEOs obtained via permutant measures. The framework is the one described in Example~\ref{exS2}, with some extensions. We consider a subset $X \subseteq \R^3$ made of points with integer coordinates, belonging to a cubic lattice and discretizing a cube $C$.
Formally, $X$ can be expressed as the product $X = \{1,\dots,n\}^3$. We consider again the group $G$ of orientation-preserving isometries that map $X$ into $X$.  As already stated in Example~\ref{exS2}, let $\pi_1,\pi_2,\pi_3$ be the three planes that contain the center of mass of $C$ and are parallel to a face of the cube $C$. Let $h_i:X\to X$ be the orthogonal symmetry with respect to $\pi_i$, for $i\in \{1,2,3\}$. Then $H_1 = \{h_1, h_2, h_3\}$ is a permutant for $G$. Let us now introduce two new permutants. First, consider the six planes $\lambda_1, \dots ,\lambda_6$ each one containing a couple of edges of $C$ that are symmetric with respect to the center of mass of $C$. Moreover, let $\ell_i:X\to X$ be the orthogonal simmetry with respect to $\lambda_i$ for $i \in \{1,2,3,4,5,6\}$. It is easy to check that $H_2 = \{\ell_1,\dots,\ell_6\}$ is a permutant for $G$. Lastly, it is trivial to verify that $H_3 = \{s\}$, where $s$ denotes the central symmetry with respect to the center of mass of the cube, is another permutant for $G$.

Given the permutants $H_1, H_2$ and $H_3$, three permutant measures $\mu_{H_1}, \mu_{H_2}, \mu_{H_3}$ on $\mathrm{Aut}(X)$ can be defined in the following way:
\[
	\mu_{H_i}(h) =
	\begin{cases}
		\frac{1}{|H_i|} & \text{if}\,\, h \in H_i,\\
		0 & \text{otherwise.}
	\end{cases}
	\quad \text{for}\,\, i \in \{1,2,3\}.
\]
By definition, we have that $\sum_{h \in \mathrm{Aut}(X)} \mu_{H_i}(h) = \sum_{h \in H_i} \mu_{H_i}(h) = 1$ for $i \in \{1,2,3\}$. Hence, in force of Theorem~\ref{mainthm2} we know that the operators defined as
\[
F_i(\varphi) = \frac{1}{|H_i|}\sum_{h\in H_i}\varphi h^{-1},\  \quad \text{for}\,\, i \in \{1,2,3\},
\]
are linear GENEOs with respect to $G$.

Now, let us introduce a classification problem that we will tackle with the help of these GENEOs. We consider functions on X that can be interpreted as 3d scans of playing dice. We denote each point $x$ of $X$ with its indexes in the grid: $x = (i,j,k)$. {All the functions $\varphi$ that we will take into account, are such that $\varphi(i,j,k) = 0$ if none of the indexes belong to the set $\{1,n\}$. This means that the functions can be non-zero only on the outer surface of the cube, as we are interested only in the visible part of the die. We observe that the number of points of $X$ that belong to the outer surface of the cube $C$ is $n^3-(n-2)^3$.
Furthermore, we model the dots on the faces of a die as two-dimensional Gaussian spots: in particular, if we consider the coordinates $(a, b)$ on a face of the cube and we center the dot at point $c = (\bar{a}, \bar{b}) \in \{1, \dots, n\}^2$, we obtain the following representation:
\[
	\varphi_{c}(a, b) =
	\begin{cases}
		e^{-{\frac{(a - \bar{a})^ 2 + (b - \bar{b})^2}{2}}} & \text{if}\,\, \max{(|a-\bar{a}|,|b -\bar{b}|)} \le 3 \\
		0  & \text{otherwise}
	\end{cases}.
\]
We set the standard deviation $\sigma = 1$, and hence $\varphi_{c}$ is close to the Gaussian function when the distance between $(a, b)$ and the center $c$ is strictly greater than $3$. Lastly, we can obtain every pattern that is displayed on a face by summing several of such functions with different centers $c_i = (\bar{a}_i, \bar{b}_i)$. Therefore, a face that shows $m$ dots located at the points $\{c_i\}_{i \in \{1,\dots,m\}}$ has the expression:}

\begin{equation}\label{face}
\varphi_{m}(a, b) = \sum_{i=1}^m k_i\, \varphi_{c_i}(a, b)
\end{equation}

{We introduce the coefficients $k_i \in [0,1]$ to model differences of intensity between the dots. These differences can be seen as noise caused by the scanning procedure.
Therefore, in the end, a data function $\varphi$ vanishes inside the cube and coincides with a function $\varphi_m$ for some $m$ and some $\{c_i\}_{i \in \{1,\dots,m\}}$ on each of the six faces of the cube.

In the following experimental part we fix $n = 25$, $n_1 = 6$, $n_2 = 13$ and $n_3 = 20$. We report some of the configurations that model the patterns of standard dice. The reader can easily guess the missing ones:
\begin{itemize}
	\item One dot: $m = 1$ and $c_1 = (n_2, n_2)$.
	\item Two dots: $m = 2$ and $c_1 = (n_1, n_1)$, $c_2 = (n_3, n_3)$.
	\item ...
	\item Five dots: $m = 5$ and $c_1 = (n_1, n_1)$, $c_2 = (n_1, n_3)$, $c_3 = (n_2, n_2)$, $c_4 = (n_3, n_1)$, $c_5 = (n_3, n_3)$.
	\item Six dots: $m = 6$ and $c_1 = (n_1, n_1)$, $c_2 = (n_1, n_2)$, $c_3 = (n_1, n_3)$, $c_4 = (n_3, n_1)$, $c_5 = (n_3, n_2)$, $c_6 = (n_3, n_3)$.
\end{itemize}

We will consider two classes of dice. The first one will contain exclusively dice in which the sum of opposite faces always equals seven: this is the class that contains the standard dice on the market. The second class is made up of dice where the sum of the number of dots on opposite faces is never equal to seven: these dice are fake since they do not exist on the market. The two classes are distinct with respect to the action of $G$, i.e., one die of the first class cannot be obtained from a die of the other through the action of a $g \in G$. Moreover, these two classes are distinct also with respect to the action of the group of all isometries of the cube (not only the ones that preserve the orientation).
\begin{figure}[htp]
	\begin{subfigure}{0.5\textwidth}
        \includegraphics[width=\textwidth]{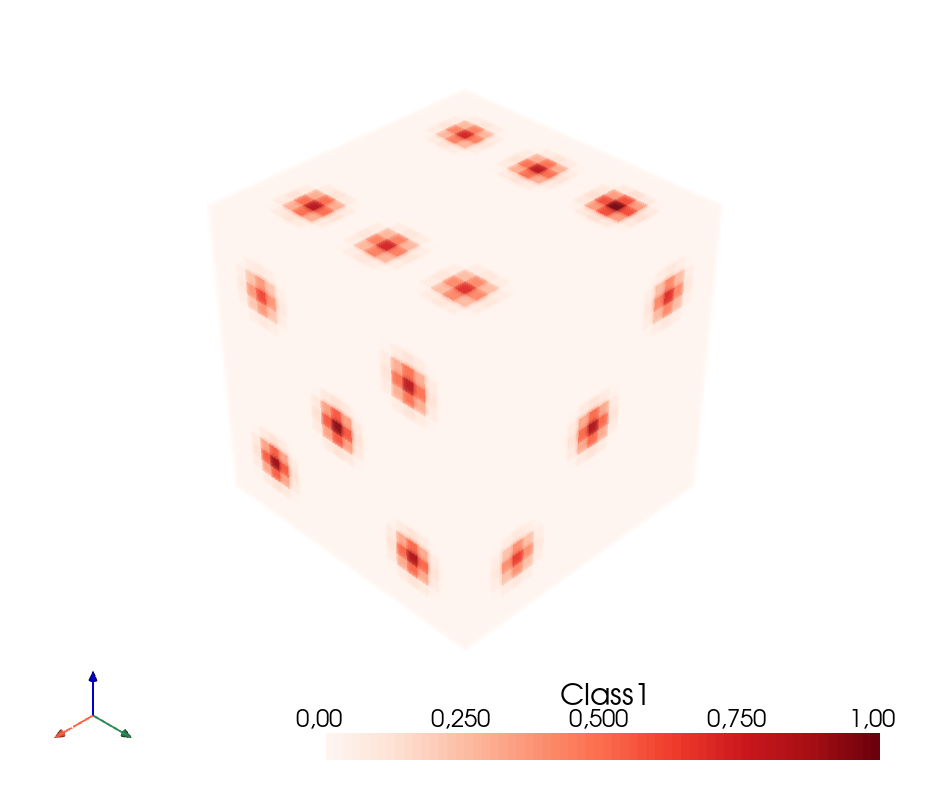}
        \caption{Firts class}
        \label{fig:class:a}
    \end{subfigure}
    \hfill
    \begin{subfigure}{0.5\textwidth}
        \includegraphics[width=\textwidth]{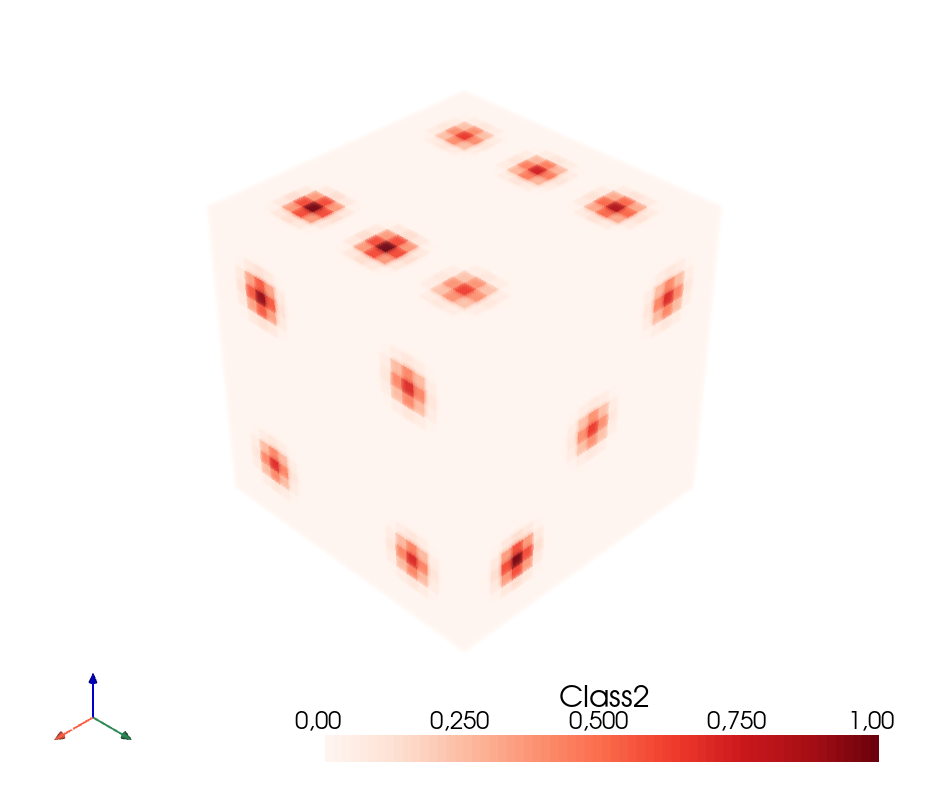}
        \caption{Second class}
        \label{fig:class:b}
    \end{subfigure}
	\caption{\textbf{Example plot of two generated functions} The left image shows a view of a die belonging to the first class, i.e., a die where the sum of the number of dots on opposite faces is always 7, as happens for the dice on the market. The right image shows a view of a die belonging to the second class, i.e., a die where the sum of the number of dots on opposite faces is never equal to 7. These dice are not regular. It is possible to see that not all the dots are equally vivid: this is due to the presence of the coefficients $k_i$. These values model the noise resulting from the scanning procedure.}
	\label{fig:exampledice}
\end{figure}

We generated a dataset of 10000 dice by means of the following procedure: dice computed at odd iterations belong to the first class, while the ones computed at even iterations belong to the second class. At each iteration, a random arrangement of the faces is obtained according to the selected class. Then, for each face, the appropriate function $\varphi_m$ is computed as the sum in~(\ref{face}) with coefficients $k_i$ drawn independently from the uniform distribution $U([0.6, 1])$.
Hence, we obtain a function $\varphi:X\to\R$.
Lastly,  a random number $p$ in the set $\{1,2,3,4,5\}$ is chosen,
and for each index $i\in\{1,\ldots,p\}$ a line $r_i$ that is a symmetry axis of the cube $C$ and is orthogonal to two of its faces is randomly chosen.
These choices are made with respect to uniform probability distributions.
Now, for each index $i$ a rotation of angle $\frac{\pi}{2}$ around $r_i$ is applied to the function $\varphi$, and hence we obtain a new function $\hat\varphi:X\to\R$ describing a die.
This ensures that all the possible spatial configurations of a die with a face placed on a tabletop can be generated.
We note that each function representing a die is characterized by $n^3-(n-2)^3=3458$ distinct values rather than the total $n^3 = 15625$ (this is due to the fact that the function vanishes inside the cube). Figure~\ref{fig:exampledice} shows two examples of functions generated by applying the previously described procedure.

In order to discriminate the data, we designed a pipeline that makes use of GENEOs, and then we compared it to the analogous pipeline that uses only the original data. The pipeline is composed as follows:

\begin{enumerate}{}	
	\item A GENEO $F$ is obtained as a convex combination of $F_1, F_2, F_3$ with weights $\alpha_1, \alpha_2, \alpha_3$, and then $F$ is applied to all the functions of the dataset. (Here we are using the property that the spaces of GENEOs are convex, provided that the spaces of data are convex~\cite{Bergomi2019}.)
	\item Each transformed function $\psi:X\to\R$ is treated as a vector $v_{\psi} \in \R^{n^3-(n-2)^3}$, storing the values of $\psi$ at the points of $X$ belonging to the outer surface of $C$. This allows to identify each die with a vector in $\R^{3458}$. To reduce the dimensionality of these data we apply the PCA (Principal Component Analysis),  preserving only the first two PCs. 
	\item With the reduced data, a SVM (Support Vector Machine) classifier with quadratic kernel is trained in order to discriminate the two classes.
\end{enumerate}

The pipeline that uses the original data skips the first step and feeds to the PCA directly the vectors associated with the original functions.

To compare the two methods, we randomly split the dataset into a training set of size 7000 and a test set of size 3000. Both the subsets have been sampled in order to maintain a balanced distribution of the two classes.
This splitting was used just to evaluate SVM classifiers in the third step, whereas the PCA was computed using the whole dataset in both cases.

\begin{figure}[htp]
	\centering
	\begin{subfigure}{0.45\textwidth}
        \includegraphics[width=\textwidth]{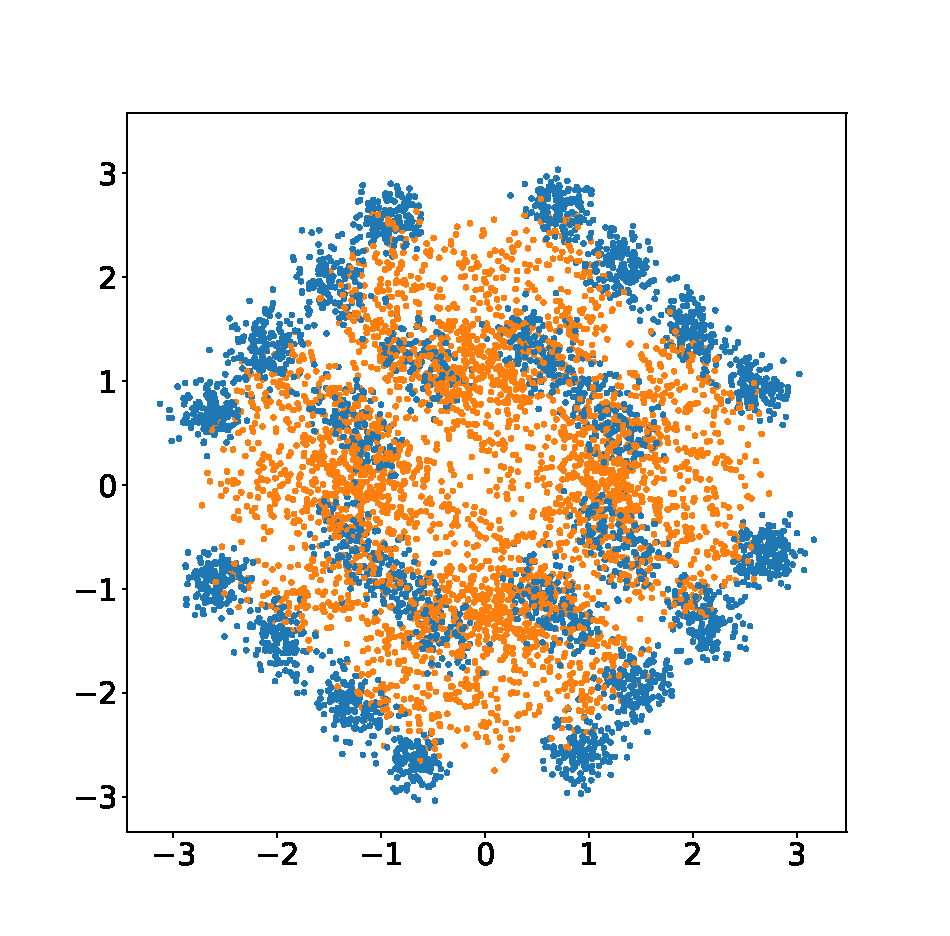}
        \caption{First method (not using GENEOs)}
        \label{fig:pca:a}
    \end{subfigure}
    \hfill
    \begin{subfigure}{0.45\textwidth}
        \includegraphics[width=\textwidth]{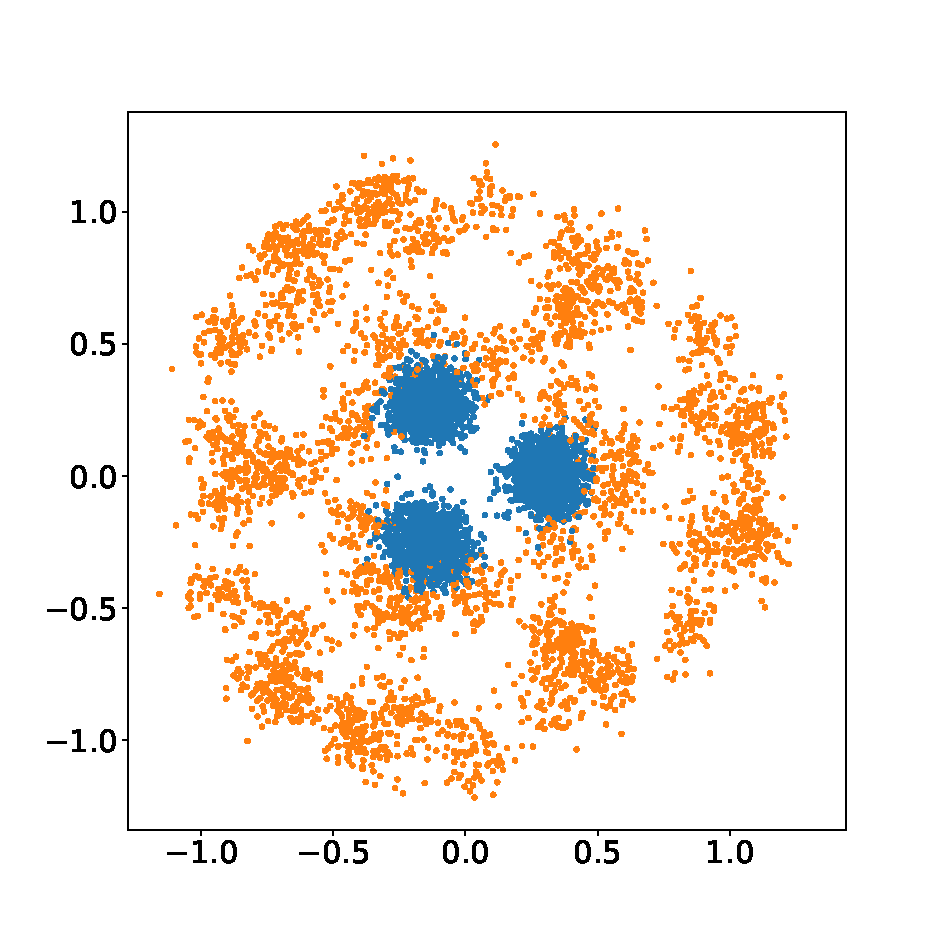}
        \caption{Second method (using GENEOs)}
        \label{fig:pca:b}
    \end{subfigure}
	\caption{\textbf{Plots of PCA results} This figure shows the results of PCA dimensionality reduction for both methods. Blue points are associated with functions belonging to the first class, while orange ones with functions belonging to the second class. A separation between the mapped points is clearly more evident in the right figure, which is the one relative to the method involving GENEOs. Hence GENEOs provide a simpler and more informative representation of the data.}
	\label{fig:pca}
\end{figure}

Figure~\ref{fig:pca} shows the results of dimensionality reduction for both kinds of data. It is clear that, for this specific choice, the outputs of the GENEO $F$ provide a clearer representation of the data: points are indeed almost separated, whereas, with the original data a separation is far less evident. However, it must be noted that not all the GENEOs obtained as convex combination of $F_1, F_2, F_3$ provide such a result. Here the parameters $\alpha_1, \alpha_2, \alpha_3$ have been repeatedly sampled from a Dirichlet distribution to select the ones giving the best results ($\alpha_1 = 0.318$, $\alpha_2 = 0.551 $, $\alpha_3 = 0.131$).

The plot also shows that in the right side picture it is reasonable to expect an almost perfect classification with a quadratic decision boundary. This fact justifies the third step of the pipeline. The decision boundaries learned by the SVM classifiers are shown in Figure~\ref{fig:svm}, while Table~\ref{tab:cms} reports confusion matrices and accuracy scores for the two methods.

\begin{figure}[htp]
	\centering
		\begin{subfigure}{0.45\textwidth}
        \includegraphics[width=\textwidth]{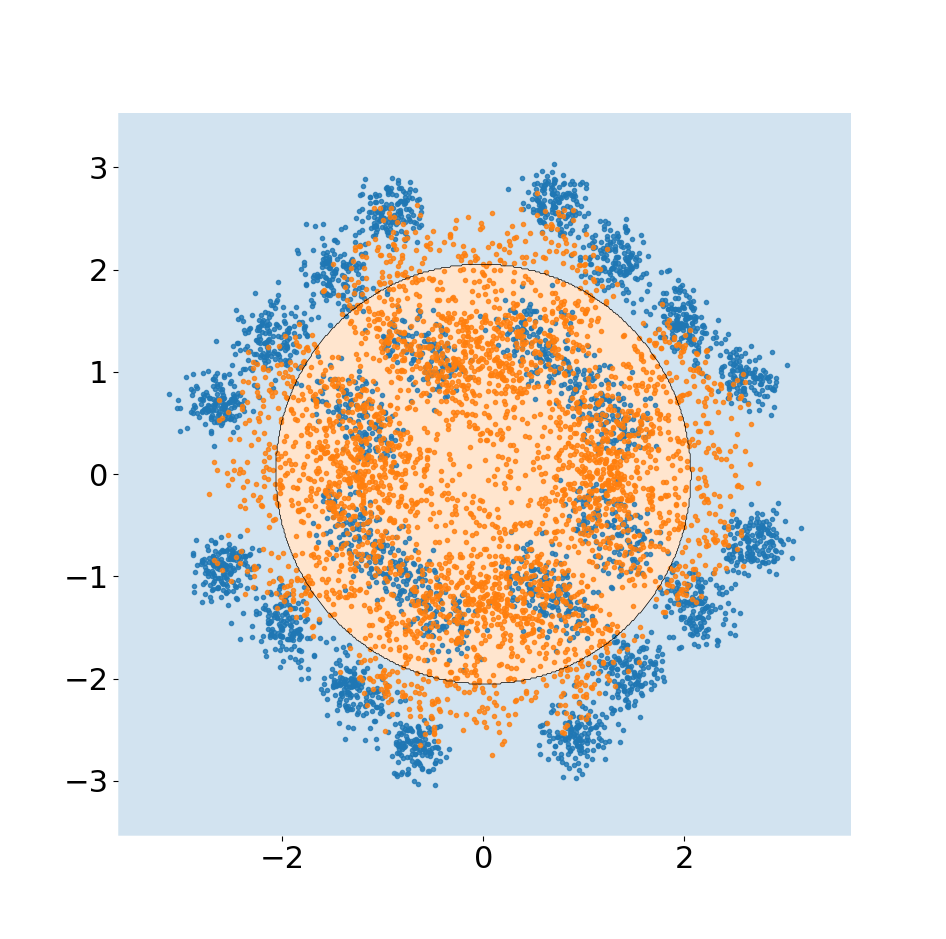}
        \caption{Firts method (not using GENEOs)}
        \label{fig:svm:a}
    \end{subfigure}
    \hfill
    \begin{subfigure}{0.45\textwidth}
        \includegraphics[width=\textwidth]{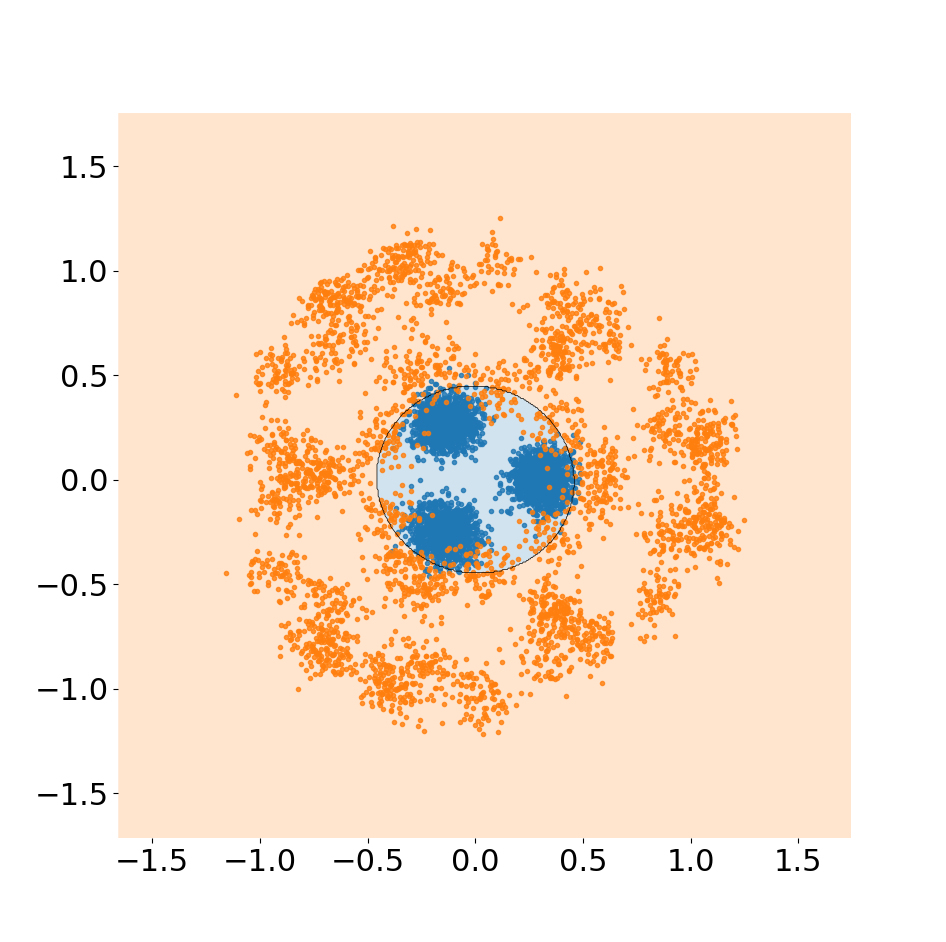}
        \caption{Second method (using GENEOs)}
        \label{fig:svm:b}
    \end{subfigure}
	\caption{\textbf{Plots of SVM decision boundaries} PCA results suggest that the simpler choice to classify the mapped points is to use a quadratic decision boundary. These figures show the boundaries learned by the SVM algorithm with a quadratic kernel for both methods. The right image shows that it is reasonable to expect better classification scores for the method involving GENEOs.}
	\label{fig:svm}
\end{figure}


\begin{table}[htp]

    \begin{subtable}[htp]{0.5\linewidth}
        \centering
        \begin{tabular}{c|c|c|c|c}
        \multicolumn{2}{c}{}&\multicolumn{2}{c}{Predicted}&\\
        \cline{3-4}
        \multicolumn{2}{c|}{}& Class 1 & Class 2 &\multicolumn{1}{c}{Total}\\
        \cline{2-4}
        \multirow{2}{*}{\rotatebox{90}{True}}& Class 1 & $2350$ & $1150$ & $3500$\\
        \cline{2-4}
        & Class 2 & $793$ & $2707$ & $3500$\\
        \cline{2-4}
        \multicolumn{1}{c}{} & \multicolumn{1}{c}{Total} & \multicolumn{1}{c}{$3143$} & \multicolumn{    1}{c}{$3857$} & \multicolumn{1}{c}{$7000$}\\
        \end{tabular}
        \caption{Train CM for original data.\\ Accuracy = 0.722}
    \end{subtable}
    \hfill
    \begin{subtable}[htp]{0.5\linewidth}
        \centering
        \begin{tabular}{c|c|c|c|c}
        \multicolumn{2}{c}{}&\multicolumn{2}{c}{Predicted}&\\
        \cline{3-4}
        \multicolumn{2}{c|}{}& Class 1 & Class 2 &\multicolumn{1}{c}{Total}\\
        \cline{2-4}
        \multirow{2}{*}{\rotatebox{90}{True}}& Class 1 & $1030$ & $470$ & $1500$\\
        \cline{2-4}
        & Class 2 & $345$ & $1155$ & $1500$\\
        \cline{2-4}
        \multicolumn{1}{c}{} & \multicolumn{1}{c}{Total} & \multicolumn{1}{c}{$1375$} & \multicolumn{    1}{c}{$1625$} & \multicolumn{1}{c}{$3000$}\\
        \end{tabular}
        \caption{Test CM for original data.\\ Accuracy = 0.728}
    \end{subtable}\\
    \begin{subtable}[htp]{0.5\linewidth}
        \centering
        \begin{tabular}{c|c|c|c|c}
        \multicolumn{2}{c}{}&\multicolumn{2}{c}{Predicted}&\\
        \cline{3-4}
        \multicolumn{2}{c|}{}& Class 1 & Class 2 &\multicolumn{1}{c}{Total}\\
        \cline{2-4}
        \multirow{2}{*}{\rotatebox{90}{True}}& Class 1 & $3432$ & $68$ & $3500$\\
        \cline{2-4}
        & Class 2 & $232$ & $3268$ & $3500$\\
        \cline{2-4}
        \multicolumn{1}{c}{} & \multicolumn{1}{c}{Total} & \multicolumn{1}{c}{$3664$} & \multicolumn{    1}{c}{$3336$} & \multicolumn{1}{c}{$7000$}\\
        \end{tabular}
        \caption{Train CM for GENEO data. \\Accuracy = 0.957}
    \end{subtable}
    \begin{subtable}[htp]{0.5\linewidth}
        \centering
        \begin{tabular}{c|c|c|c|c}
        \multicolumn{2}{c}{}&\multicolumn{2}{c}{Predicted}&\\
        \cline{3-4}
        \multicolumn{2}{c|}{}& Class 1 & Class 2 &\multicolumn{1}{c}{Total}\\
        \cline{2-4}
        \multirow{2}{*}{\rotatebox{90}{True}}& Class 1 & $1464$ & $36$ & $1500$\\
        \cline{2-4}
        & Class 2 & $100$ & $1400$ & $1500$\\
        \cline{2-4}
        \multicolumn{1}{c}{} & \multicolumn{1}{c}{Total} & \multicolumn{1}{c}{$1564$} & \multicolumn{    1}{c}{$1436$} & \multicolumn{1}{c}{$3000$}\\
        \end{tabular}
        \caption{Test CM for GENEO data.\\ Accuracy = 0.955}
    \end{subtable}
    \caption{\textbf{Confusion Matrices} These tables report confusion matrices for both methods and for both train and test set. The accuracy score is considerably higher for the method which uses GENEOs both for training and test set. Since train and test accuracy scores are very near for both methods, we are confident that the methods perform equally well on unseen data, meaning that they are not overfitting.}
    \label{tab:cms}
\end{table}

As we expected, the method employing GENEOs has an accuracy score of 0.955 on the test set, a value significantly higher than the one of the method without GENEOs, which is 0.728.

This result holds even if we modify some of the hyperameters. For example, we can change the distribution of the coefficients $k_i$. In particular, we can consider again a uniform distribution but with a smaller range containing 1 (i.e., [0.8, 1.0]). In this way we reduce the randomness of the data. Because of this, the plots of Figures~\ref{fig:pca} and~\ref{fig:svm} tend to show groups of points that are much more concentrated, and therefore both the methods perform better, even though there is still a gap between the two. Furthermore, if we change the number of PCs retained by the PCA, we observe similar results by keeping one to three PCs. Starting from four PCs onward the two methods perform almost equally well. Another possibility is to use a different kernel for the SVM classifiers. For example, one can use a Radial Basis Function (RBF) kernel to obtain more complex decision boundaries.
Models capable of learning more complex decision boundaries reduce the gap of performances between the two approaches, nonetheless considering up to two PCs the use of GENEOs is still highly beneficial. A summary of these experiments is reported in Table~\ref{tab:results}.

\begin{table}[htp]
\begin{center}
    \begin{tabular}{ |c|c|c|c|c| }
    \hline
    $k_i$ distribution & \makecell{Number \\ of PCs} & SVM kernel & \makecell{Accuracy \\ without GENEOs} & \makecell{Accuracy \\ with GENEOs}\Tstrut\Bstrut\\
    \hline
    $U([0.8, 1.0])$ & 1 & quadratic & 0.580 & \textbf{0.839}\Tstrut\\
    $U([0.8, 1.0])$ & 2 & quadratic & 0.726 & \textbf{0.999}\\
    $U([0.8, 1.0])$ & 3 & quadratic & 0.974 & \textbf{0.999}\\
    $U([0.8, 1.0])$ & 4 & quadratic & 0.981 & \textbf{0.999}\Bstrut\\
    \hline
    $U([0.8, 1.0])$ & 1 & RBF       & 0.685 & \textbf{0.915}\Tstrut\\
    $U([0.8, 1.0])$ & 2 & RBF       & 0.916 & \textbf{1.000}\\
    $U([0.8, 1.0])$ & 3 & RBF       & \textbf{1.000} & \textbf{1.000}\Bstrut\\
    \hline
    $U([0.6, 1.0])$ & 1 & quadratic & 0.589 & \textbf{0.819}\Tstrut\\
    $U([0.6, 1.0])$ & 2 & quadratic & 0.728 & \textbf{0.955}\\
    $U([0.6, 1.0])$ & 3 & quadratic & 0.915 & \textbf{0.956}\\
    $U([0.6, 1.0])$ & 4 & quadratic & 0.930 & \textbf{0.955}\Bstrut\\
    \hline
    $U([0.6, 1.0])$ & 1 & RBF       & 0.615 & \textbf{0.828}\Tstrut\\
    $U([0.6, 1.0])$ & 2 & RBF       & 0.816 & \textbf{0.974}\\
    $U([0.6, 1.0])$ & 3 & RBF       & 0.970 & \textbf{0.976}\Bstrut\\
    \hline
    $U([0.4, 1.0])$ & 1 & quadratic & 0.591 & \textbf{0.778}\Tstrut\\
    $U([0.4, 1.0])$ & 2 & quadratic & 0.718 & \textbf{0.902}\\
    $U([0.4, 1.0])$ & 3 & quadratic & 0.860 & \textbf{0.901}\\
    $U([0.4, 1.0])$ & 4 & quadratic & 0.881 & \textbf{0.903}\Bstrut\\
    \hline
    $U([0.4, 1.0])$ & 1 & RBF       & 0.600 & \textbf{0.780}\Tstrut\\
    $U([0.4, 1.0])$ & 2 & RBF       & 0.742 & \textbf{0.910}\\
    $U([0.4, 1.0])$ & 3 & RBF       & 0.893 & \textbf{0.909}\\
    $U([0.4, 1.0])$ & 4 & RBF       & \textbf{0.932} & 0.911\Bstrut\\
    \hline
    \end{tabular}
    \caption{\textbf{Accuracy scores for different hyperparameters} This table shows a comparison between the accuracies of the two methods for different combinations of the hyperparameters. Using GENEOs is always beneficial when combined with a quadratic kernel SMV and a PCA with number of PCs less than four.
Otherwise, with a RBF kernel SVM, GENEOs are beneficial up to two or three PCs, whereas retaining more PCs the two methods tend to perform equally well.}
    \label{tab:results}
\end{center}
\end{table}

\section{Discussion}\label{discussion}
In our paper we have proved that all linear GEOs and GENEOs can be produced by means of a dual method based on the concept of permutant measure with respect to a group $G$, under the assumption that $G$  transitively acts on a finite set. This method could be particularly useful when we have to deal with a large group $G$,
as frequently happens in real applications. Summations on large groups can indeed present computational difficulties,
while summations on the supports of permutant measures are often easier.
The use of the set of all permutant measures also benefits of its lattice structure.
The availability of the approach we have studied in this paper could be relevant for
the application of GEOs and GENEOs as multi-level components in deep learning and make the construction of neural networks more
transparent and interpretable, according to the mathematical framework proposed in~\cite{Bergomi2019}. The next natural step in this line of research is the extension of our approach to topological groups. We plan to study this possible extension in a forthcoming paper.

{In Section~\ref{example} we have also illustrated an example, showing how GENEOs built by permutant measures could be used in order to extract relevant properties from data.
In our opinion, GENEOs of this kind could be of great use in machine learning, since they can inject information about the way data should be managed on the basis of prior knowledge. In the given example, this knowledge was represented by the clear relevance of symmetries in data concerning a cube, leading us to focus our attention on
symmetry planes.
We are aware of the toy nature of this example, nonetheless we believe that it is important for two main reasons: first, it shows that a small network of GENEOs can be more informative than its single components, secondly, it makes clear that GENEOs allow to obtain simple and explainable representations of data that are easier to process in a pipeline with other explainable methods.
We know that this specific example could be managed in many other ways, without the drastic dimensionality reduction of PCA and with methods more complex than SVM classifiers. Despite this, we showed that GENEOs can improve the results by requiring only minimal information. In the near future, we plan to develop more extended applications of GENEOs built via permutant measures, employing real world data.}

\begin{acknowledgements}
This research has been partially supported by INdAM-GNSAGA.
The authors thank Alessandro Achille, Fabio Anselmi and Giovanni Paolini for their helpful advice.
This paper is dedicated to the memory of Dr. Mohamed Mashally.
\end{acknowledgements}

%
 \section*{Conflict of interest}

 The authors declare that they have no conflict of interest.

 \section*{Authors' contributions} P.F. devised the project. All authors contributed to the manuscript. All authors read and approved the final manuscript.
 The authors of this paper have been listed in alphabetical order.



\bibliographystyle{spmpsci}
\bibliography{bibGP}

\end{document}